\documentclass[11pt,reqno]{amsart}
\usepackage[utf8]{inputenc}

\usepackage{amssymb, amsmath, amsthm}
\usepackage{hyperref}
\usepackage[alphabetic,lite]{amsrefs}
\usepackage{verbatim}
\usepackage{amscd}   % for commutative diagrams
\usepackage[all]{xy} % for complicated commutative diagrams
\usepackage{youngtab} % for Young tableaux
\usepackage{young} % for Young tableaux
\usepackage{ytableau}
\usepackage{tikz}
\usepackage{ mathrsfs }
\usepackage{cases}
\usepackage{array}
\usepackage{cellspace}
\usepackage{calligra,mathrsfs}
\usepackage{bm}
\usepackage{graphicx}
\usepackage{rank-2-roots}
\usepackage{float}
\usepackage{enumitem}

\newcommand{\defi}[1]{{\upshape\sffamily #1}}
\DeclareMathOperator{\ShHom}{\mathscr{H}\text{\kern -3pt {\calligra\large om}}\,}

\newcommand{\bw}{\bigwedge}

\def\kk{{\mathbf k}}
\renewcommand{\ll}{\lambda}

\newcommand{\onto}{\twoheadrightarrow}
\newcommand{\oo}{\otimes}

\newcommand{\GL}{\operatorname{GL}}

\newcommand{\Proj}{\operatorname{Proj}}

\newcommand{\rk}{\operatorname{rank}}
\newcommand{\SL}{\operatorname{SL}}

\newcommand{\Sym}{\operatorname{Sym}}

\newcommand{\coker}{\operatorname{coker}}
\renewcommand{\det}{\operatorname{det}}

\renewcommand{\ker}{\operatorname{ker}}

\newcommand{\reg}{\operatorname{reg}}

\newcommand{\bb}[1]{\mathbb{#1}}

\renewcommand{\rm}[1]{\textrm{#1}}
\newcommand{\mc}[1]{\mathcal{#1}}
\newcommand{\mf}[1]{\mathfrak{#1}}

\newcommand{\op}[1]{\operatorname{#1}}

\newcommand{\ul}[1]{\underline{#1}}

\def\PP{{\mathbf P}}
\def\lra{\longrightarrow}

\newtheorem{theorem}{Theorem}[section]
\newtheorem*{theorem*}{Theorem}
\newtheorem*{problem*}{Problem}
\newtheorem{lemma}[theorem]{Lemma}

\newtheorem{proposition}[theorem]{Proposition}
\newtheorem{corollary}[theorem]{Corollary}
\newtheorem*{corollary*}{Corollary}

\newtheorem*{main-thm*}{Main Theorem}
\newtheorem*{linear-resolutions*}{Theorem on Linear Resolutions}
\newtheorem*{regularity-powers*}{Theorem on Regularity}
\newtheorem*{injectivity-Ext*}{Theorem on Injectivity of Maps of Ext Modules}
\newtheorem*{Kodaira*}{Kodaira Vanishing for Determinantal Thickenings}

\theoremstyle{definition}

\newtheorem*{definition*}{Definition}
\newtheorem{example}[theorem]{Example}

\theoremstyle{remark}
\newtheorem{remark}[theorem]{Remark}
\newtheorem*{remark*}{Remark}

\numberwithin{equation}{section}

\begin{document}

\title{Cohomology of line bundles on the incidence correspondence}

\author{Zhao Gao}
\address{Department of Mathematics, University of Notre Dame, 255 Hurley, Notre Dame, IN 46556}
\email{zgao1@nd.edu}

\author{Claudiu Raicu}
\address{Department of Mathematics, University of Notre Dame, 255 Hurley, Notre Dame, IN 46556\newline
\indent Institute of Mathematics ``Simion Stoilow'' of the Romanian Academy}
\email{craicu@nd.edu}


\date{\today}


\begin{abstract} 
 For a finite dimensional vector space $V$ of dimension $n$, we consider the incidence correspondence (or partial flag variety) $X\subset\bb{P}V \times \bb{P}V^{\vee}$, parametrizing pairs consisting of a point and a hyperplane containing it. We completely characterize the vanishing and non-vanishing behavior of the cohomology groups of line bundles on $X$ in characteristic $p>0$. If $n=3$ then $X$ is the full flag variety of $V$, and the characterization is contained in the thesis of Griffith from the 70s. In characteristic $0$, the cohomology groups are described for all $V$ by the Borel--Weil--Bott theorem. Our strategy is to recast the problem in terms of computing cohomology of (twists of) divided powers of the cotangent sheaf on projective space, which we then study using natural truncations induced by Frobenius, along with careful estimates of Castelnuovo--Mumford regularity. When $n=3$, we recover the recursive description of characters from recent work of Linyuan Liu, while for general $n$ we give character formulas for the cohomology of a restricted collection of line bundles. Our results suggest truncated Schur functions as the natural building blocks for the cohomology characters.
\end{abstract}

\maketitle

\section{Introduction}\label{sec:intro}

A fundamental problem at the interface of algebraic geometry and representation theory is to determine the sheaf cohomology groups of line bundles on (partial) flag varieties, or at the very least, to characterize the vanishing and non-vanishing behavior of these groups \cite{and-survey}. In characteristic zero, a complete answer is given by the Borel--Weil--Bott theorem. Despite significant results such as Kempf's vanishing theorem \cite{kempf}, or Andersen's characterization of the (non-)vanishing of $H^1$ \cite{andersen}, over fields of positive characteristic the problem remains largely unanswered. The main goal of this paper is to characterize the (non-)vanishing of cohomology of line bundles on the \defi{incidence correspondence} (see also \cites{liu-coker,liu-polo})
\begin{equation}\label{eq:def-X}
    X = \{(p,H) \in \bb{P}V \times \bb{P}V^{\vee} : p\in H\},
\end{equation}
where $V$ is a vector space of dimension $n\geq 3$ over an algebraically closed field $\kk$ of characteristic $p>0$. When $n=3$ we give an explicit recursive relation describing the characters of the cohomology groups with respect to the natural action of the $n$-dimensional torus, recovering a recent result of Linyuan Liu \cite{liu-thesis} (see also \cites{donkin,fazeel} for a different set of recursions). When $n=3$ and $p=2$, we also give a non-recursive character formula, involving symmetric polynomials naturally defined in terms of winning positions in the game Nim.

\medskip

\noindent{\bf (Non-)vanishing of line bundle cohomology.} Using \eqref{eq:def-X}, we have an inclusion $\iota:X\hookrightarrow\bb{P}V \times \bb{P}V^{\vee}$ which induces an isomorphism on the Picard groups, hence every line bundle on $X$ can be expressed uniquely as
\[\mc{O}_X(a,b) = \iota^*\left(\mc{O}_{\bb{P}V \times \bb{P}V^{\vee}}(a,b)\right)\text{ for }a,b\in\bb{Z}.\]
The following is the main result of our paper, and as we explain shortly, it leads to a complete characterization of the vanishing and non-vanishing behavior of the cohomology of line bundles on $X$.

\begin{theorem}\label{thm:main}
 Suppose that $b\leq -n$, $a\geq -b-n+1$, and consider the line bundle $\mc{L}=\mc{O}_X(a,b)$. If we define $d=-b-n+1$, then by considering the $p$-adic expansion of $d$, we can find unique integers $t,k$ such that
 \[ tp^k \leq d<(t+1)p^k,\ 1\leq t<p,\ k\geq 0.\]
 The following characterize the (non-)vanishing behavior for the cohomology of $\mc{L}$.
 \begin{enumerate}
     \item $H^{n-1}(X,\mc{L})=0$ if and only if $a\geq(t+n-2)p^k-n+2$.
     \item $H^{n-2}(X,\mc{L})=0$ if and only if $n=3$ and $a=d=(t+1)p^k-1$.
     \item $H^i(X,\mc{L})=0$ for all $i\not\in\{n-1,n-2\}$.
 \end{enumerate}
\end{theorem}

To motivate the restrictions on $a,b$ in Theorem~\ref{thm:main}, we make the following observations (see also Section~\ref{sec:prelims}). As in Theorem~\ref{thm:main}, we write $\mc{L}=\mc{O}_X(a,b)$.
\begin{itemize}
    \item A special case of Kempf's vanishing theorem implies that the line bundle $\mc{L}$ has global sections if and only if $a,b\geq 0$, in which case $H^i(X,\mc{L})=0$ for all $i>0$.
    \item An easy argument shows that if $-n+2\leq a\leq -1$ or $-n+2\leq b\leq -1$ then $H^i(X,\mc{L})=0$ for all $i$.
    \item We have $\dim X=2n-3$, and the canonical bundle on $X$ is
    \[\omega_X = \mc{O}_X(-n+1,-n+1),\]
    hence by Serre duality we obtain
    \[ H^i(X,\mc{L}) = H^{2n-3-i}\left(X,\mc{O}_X(-n+1-a,-n+1-b)\right)^{\vee}.\]
    In particular, if $a,b\leq -n+1$ then by combining Serre duality and Kempf's vanishing, we conclude that $H^{2n-3}(X,\mc{L})\neq 0$, and $H^i(X,\mc{L})=0$ for $i\neq 2n-3$.
    \item By exchanging the roles of $V$ and $V^{\vee}$, we may replace $\mc{O}_X(a,b)$ with $\mc{O}_X(b,a)$, and based on the previous observations we may assume that $a\geq 0$ and $b\leq -n+1$.
    \item If $a\geq 0$, $b\leq -n+1$, and $a+b\leq -n+1$, then we can first exchange $\mc{O}_X(a,b)$ with $\mc{O}_X(b,a)$ and then apply Serre duality to conclude that
    \begin{equation}\label{eq:Serre+swap}
    H^i(X,\mc{L})\neq 0 \Longleftrightarrow H^{2n-3-i}(X,\mc{O}_X(-n+1-b,-n+1-a))\neq 0.
    \end{equation}
    Noting that $-n+1-b\geq 0$, $-n+1-a\leq -n+1$, and $(-n+1-b)+(-n+1-a)\geq -n+1$, we are reduced to studying line bundles $\mc{L}$ for which $a\geq -b-n+1\geq 0$.
    \item If $b=-n+1$ then using \eqref{eq:Serre+swap} and the equality
    \[H^{2n-3-i}(X,\mc{O}_X(0,-n+1-a)) = H^{2n-3-i}(\bb{P}V^{\vee},\mc{O}_{\bb{P}V^{\vee}}(-n+1-a)),\]
    we conclude that for $a\geq 0$ and $b=-n+1$ we have
    \[H^i(X,\mc{L})\neq 0 \Longleftrightarrow i=n-2.\]
\end{itemize}

To provide more context to Theorem~\ref{thm:main}, consider the complete flag variety $\bb{F}(V)$, let $\mc{Q}_i$ denote the tautological rank $i$ quotient sheaf on $\bb{F}(V)$, and let
\[ \mc{L}_i = \ker\left(\mc{Q}_i \onto \mc{Q}_{i-1}\right)\text{ for }i=1,\cdots,n.\]
The Picard group of $\bb{F}(V)$ is isomorphic to $\bb{Z}^{n-1}=\bb{Z}^n/\bb{Z}(1,\cdots,1)$, generated by $\mc{L}_1,\cdots,\mc{L}_n$, subject to the relation
\[\mc{L}_1\oo\cdots\oo\mc{L}_n \simeq \bw^n V \oo \mc{O}_{\bb{F}(V)} \simeq \mc{O}_{\bb{F}(V)}.\]
We write
\[\mc{O}_{\bb{F}(V)}(\lambda) = \mc{L}_1^{\lambda_1}\oo\cdots\oo \mc{L}_n^{\lambda_n}\text{ for }\lambda\in\bb{Z}^n.\]
One can identify the incidence correspondence $X$ with a partial flag variety, and as such there exists a natural forgetful map
\[ \pi:\bb{F}(V) \lra X,\]
obtained by forgetting the intermediate elements of a flag (except those in dimension and codimension one). In particular, $\pi$ is an isomorphism when $\dim(V)=3$, in which case Theorem~\ref{thm:main} is equivalent to the main result of \cite{griffith}. In general, we have
\[ \mc{O}_X(a,b) = \pi^*\left(\mc{O}_{\bb{F}(V)}(a,0,\cdots,0,-b)\right)\quad\text{ for all }a,b\in\bb{Z}.\]
Moreover, we have
\[\pi_*\mc{O}_{\bb{F}(V)}=\mc{O}_X\quad\text{ and }\quad R^j\pi_*\mc{O}_{\bb{F}(V)}=0\text{ for }j>0,\]
and it follows from the projection formula and the Leray spectral sequence that
\[ H^j(X,\mc{O}_X(a,b)) = H^j(\bb{F}(V),\mc{O}_{\bb{F}(V)}(a,0,\cdots,0,-b))\text{ for all }j,a,b.\]
We can then interpret Theorem~\ref{thm:main} as characterizing the (non-)vanishing of cohomology of the line bundles $\mc{O}_{\bb{F}(V)}(\lambda)$ when $\lambda_2=\cdots=\lambda_{n-1}=0$. Using a more Lie theoretic perspective, we can identify $\op{Pic}(X)$ with the rank $2$ slice of $\op{Pic}(\bb{F}(V))$ generated by the fundamental weights
\[ \omega_1=(1,0,\cdots,0)\quad\text{and}\quad \omega_{n-1}=(1,\cdots,1,0)\equiv(0,\cdots,0,-1),\]
and via this identification $\mc{O}_X(a,b)$ corresponds to $a\omega_1+b\omega_{n-1}$.

\begin{example}\label{ex:n=4-p=3} To illustrate Theorem~\ref{thm:main} and the subsequent discussion, we consider the following diagram depicting the case when $n=4$ and $p=3$. 
%\iffalse
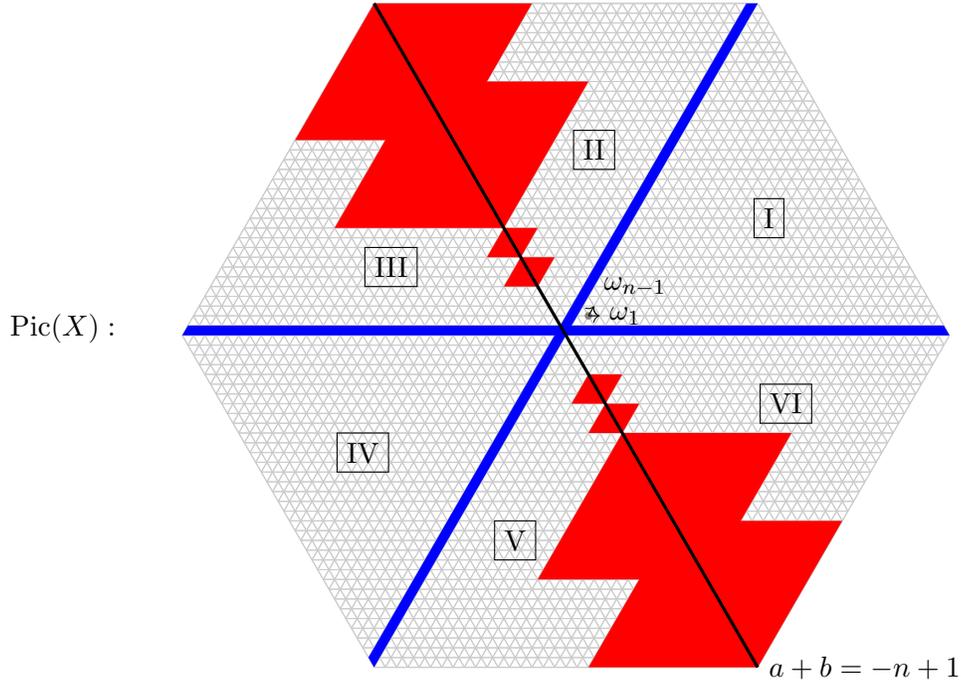
\begin{figure}[H]
\[\op{Pic}(X):\qquad
\begin{tikzpicture}[baseline=(O.base)]
\setlength\weightLength{.15cm}%{.22cm}
\begin{rootSystem}{A}
\weightLattice{34}
\node (O) at \weight{0}{0} {};
\wt{1}{2}
\draw[->] \weight{1}{2} -- \weight{2}{2} 
    node[right]{$\omega_1$};
\draw[->] \weight{1}{2} -- \weight{1}{3} 
    node[above right]{$\omega_{n-1}$};
%\draw[very thick] \weight{-20}{0} -- \weight{20}{0};
%\draw[very thick] \weight{-20}{1} -- \weight{19}{1};
\fill[blue!100] \weight{-34}{0} -- \weight{34}{0} -- \weight{33}{1} -- \weight{-34}{1} -- cycle;
\fill[blue!100] \weight{0}{-34} -- \weight{0}{34} -- \weight{-1}{34} -- \weight{-1}{-33} -- cycle;
\fill[red!100] \weight{4}{-4} -- \weight{7}{-4} -- \weight{7}{-7} -- \weight{4}{-7} -- cycle;
\fill[red!100] \weight{7}{-7} -- \weight{10}{-7} -- \weight{10}{-10} -- \weight{7}{-10} -- cycle;
\fill[red!100] \weight{10}{-10} -- \weight{25}{-10} -- \weight{25}{-25} -- \weight{10}{-25} -- cycle;
\fill[red!100] \weight{19}{-19} -- \weight{34}{-19} -- \weight{34}{-34} -- \weight{19}{-34} -- cycle;
\fill[red!100] \weight{-5}{5} -- \weight{-8}{5} -- \weight{-8}{8} -- \weight{-5}{8} -- cycle;
\fill[red!100] \weight{-8}{8} -- \weight{-11}{8} -- \weight{-11}{11} -- \weight{-8}{11} -- cycle;
\fill[red!100] \weight{-11}{11} -- \weight{-26}{11} -- \weight{-26}{26} -- \weight{-11}{26} -- cycle;
\fill[red!100] \weight{-20}{20} -- \weight{-34}{20} -- \weight{-34}{34} -- \weight{-20}{34} -- cycle;
\node[draw] at \weight{12}{12} {I};
\node[draw] at \weight{-7}{19} {II};
\node[draw] at \weight{-19}{7} {III};
\node[draw] at \weight{-12}{-12} {IV};
\node[draw] at \weight{6}{-21} {V};
\node[draw] at \weight{23}{-7} {VI};
\draw[very thick] \weight{-34}{34} -- \weight{34}{-34} node[right]{$a+b=-n+1$};
\end{rootSystem}
\end{tikzpicture}
\]
\caption{The Picard group of the incidence correspondence, (non-)vanishing behavior for cohomology}
\label{fig:PicX}
\end{figure}
%\fi

The weights contained in the blue region correspond to line bundles with vanishing cohomology (when either $a$ or $b$ is in $\{-1,-2,\cdots,-n+2\}$). Region $\op{I}$ gives the effective cone, where Kempf's vanishing theorem implies that all line bundles have vanishing higher cohomology. Serre duality exchanges regions $\op{I}\leftrightarrow \op{IV}$, $\op{II}\leftrightarrow \op{V}$, $\op{III}\leftrightarrow \op{VI}$, while the duality interchanging the vector space $V$ with $V^{\vee}$ fixes regions $\op{I}$ and $\op{IV}$, and exchanges $\op{II}\leftrightarrow \op{VI}$, $\op{III}\leftrightarrow \op{V}$. Theorem~\ref{thm:main} is then entirely focused on region $\op{VI}$, and by applying the dualities (and Kempf's theorem), one can deduce the results for the remaining regions. By the Borel--Weil--Bott theorem, over a field of characteristic zero the regions $\op{II}$ and $\op{VI}$ correspond to line bundles whose only non-vanishing cohomology is $H^{n-2}$, while for regions $\op{III}$ and $\op{V}$ the only non-vanishing cohomology is $H^{n-1}$. The red region indicates all the line bundles which in characteristic $p=3$ have $H^{n-2}\neq 0\neq H^{n-1}$, and have vanishing cohomology in all other degrees. For line bundles not in the red region, the dimensions (and characters) of all the cohomology groups are the same as in characteristic zero.
\end{example}

\medskip

\noindent{\bf Divided powers of the cotangent sheaf.} To give yet another perspective on $X$ and the cohomology of its line bundles, we consider the projective space $\PP = \bb{P}V = \op{Proj}(\Sym V) \simeq \PP^{n-1}$ (parametrizing $1$-dimensional quotients of $V$ or equivalently, $1$-dimensional subspaces of $V^{\vee}$). We consider the tautological sequence
\begin{equation}\label{eq:taut-ses-PV}
0 \lra \mc{R} \lra V\oo\mc{O}_{\PP} \lra \mc{Q} \lra 0,
\end{equation}
where $\mc{Q}=\mc{O}_{\PP}(1)$ is the tautological line bundle on $\PP$, and $\mc{R}\simeq\Omega^1_{\PP}(1)$ is the tautological subsheaf. We have that $X = \bb{P}\left(\mc{R}^{\vee}\right)$ is a projective bundle over $\PP$, and if we let $D^d$ denote the $d$-th \defi{divided power functor}, then the following holds.

\begin{theorem}\label{thm:FV-to-PV}
 Suppose that $b\leq -n$, $a\geq -b-n+1$, and let $d=-b-n+1$ and $\mc{L}=\mc{O}_X(a,b)$ as in Theorem~\ref{thm:main}. We have
 \begin{equation}\label{eq:Hj-X-vs-P}
 H^j(X,\mc{L}) \simeq H^{j-n+2}(\PP,D^d\mc{R}(a-1)).
 \end{equation}
 Moreover, for $d\geq 1$ and $t,k$ as in Theorem~\ref{thm:main}, the Castelnuovo--Mumford regularity of $D^d\mc{R}$ is given by
 \begin{equation}\label{eq:reg-DdR}
 \op{reg}(D^d\mc{R}) = (t+n-2)\cdot p^k-n+2.
 \end{equation}
\end{theorem}

To explain the last part of Theorem~\ref{thm:FV-to-PV}, and relate it to Theorem~\ref{thm:main}, we recall that a coherent sheaf $\mc{F}$ on projective space $\PP$ is said to be \defi{$m$-regular} (see for instance \cite{lazarsfeld-pos-I}*{Section~I.8}) if 
\begin{equation}\label{eq:def-m-reg}
H^i(\PP,\mc{F}(m-i))=0 \text{ for all }i>0.
\end{equation}
If $\mc{F}$ is $m$-regular, then it is also $m'$-regular for all $m'\geq m$. The \defi{(Castenuovo--Mumford) regularity} of $\mc{F}$ is then
\begin{equation}\label{eq:def-CM-reg}
 \reg(\mc{F}) = \min\{m : \mc{F}\text{ is $m$-regular}\}.
\end{equation}
Combining \eqref{eq:Hj-X-vs-P} with \eqref{eq:reg-DdR} and \eqref{eq:def-m-reg}, it follows that
\[H^{n-1}(X,\mc{L}) \simeq H^1(\PP,D^d\mc{R}(a-1))=0\text{ for }a\geq (t+n-2)\cdot p^k-n+2,\]
which is conclusion (1) of Theorem~\ref{thm:main}. 

\medskip

\noindent{\bf An extremal example.} It is elementary to verify that
\[ H^i\left(\PP,D^d\mc{R}(a-1)\right) = 0\text{ for }d,a\geq 0\text{ and }i\geq 2,\]
hence the assertion that $D^d\mc{R}$ is not $a$-regular for $a=(t+n-2)\cdot p^k-n+1$ is equivalent to the non-vanishing of $H^1(\PP,D^d\mc{R}(a-1))$. When $d=tp^k$ we have a precise description of this cohomology group (corresponding to a ``peak'' in the ``mountain range'' from region VI of Figure~\ref{fig:PicX}). We write $F^{p^k}$ for the subfunctor of $\Sym^{p^k}$ defined for each vector space $W$ by
\[ F^{p^k}W = \op{Span}_{\kk}\{w^{p^k} : w \in W\} \subseteq \Sym^{p^k}W.\]

\begin{theorem}\label{thm:corner-cohomology}
 For $1\leq t<p$ and $k\geq 0$, if we let 
 \[a=(t+n-2)p^k-n+1,\quad d=tp^k,\quad b=-d-n+1,\]
 then we have
 \[H^{n-1}\left(X,\mc{O}_X(a,b)\right) = H^1\left(\PP,D^{d}\mc{R}(a-1)\right)=F^{p^k}\left(\bb{S}_{(t-1,t-1)}V\right),\]
 which is the simple $\SL$-module of highest weight $p^k(t-1)\omega_2$. 
\end{theorem}

\begin{remark}\label{rem:simplicity}
  Since the highest weight of $\bb{S}_{(t-1,t-1)}V$ is $(t-1)\omega_2$, it follows that the highest weight of $F^{p^k}\left(\bb{S}_{(t-1,t-1)}V\right)$ is $p^k(t-1)\omega_2$, so the last assertion of Theorem~\ref{thm:corner-cohomology} is about the simplicity of the displayed module. Since the Frobenius functor preserves simplicity \cite{humphreys}*{Section~2.7}, it suffices to explain why $\bb{S}_{(t-1,t-1)}V$ is simple. This is equivalent to the simplicity of the corresponding Weyl module $\left(\bb{S}_{(t-1,t-1)}(V^{\vee})\right)^{\vee}$, which follows for instance from \cite{james-decomp}*{Corollary~3.5}, \cite{james-conj-carter}*{Corollary~2.4} and the fact that none of the hook lengths of the Young diagram of shape $(t-1,t-1)$ has size divisible by $p$. Our contribution is therefore just the explicit calculation of the displayed cohomology group.
\end{remark}

\medskip

\noindent{\bf Character calculations.} We turn our attention to the question of computing characters of the $\SL$-representations underlying the cohomology groups discussed so far. Our results will single out \defi{truncated Schur functions} as the main building blocks for cohomology. We fix an identification $V\simeq\kk^n$ and write
\[ A=\bb{Z}[x_1,\cdots,x_n]/\langle x_1\cdots x_n-1\rangle\]
for the representation ring of the maximal torus $(\kk^{\times})^n/(\kk^{\times})$ in  $\SL_n\simeq\SL(V)$. We write $[W]$ for the character of a finite-dimensional representation $W$, and note that the endomorphism ${}^{\vee}:A\lra A$ defined by $x_i\mapsto x_i^{-1}$ has the property that
\[ [W^{\vee}] = [W]^{\vee}\text{ for all }W.\]
Similarly, the endomorphism $F^p:A\lra A$ sending $x_i \lra x_i^p$ satisfies
\[ [F^p W] = F^p[W]\text{ for all }W.\]
We consider the \defi{elementary symmetric functions} $e_d$, the \defi{complete symmetric functions} $h_d$ and the \defi{modular (or truncated) complete symmetric functions} $h'_d$, defined by (see also Section~\ref{subsec:trunc-div-sym})
\[ e_d = \left[\bw^d V\right],\quad h_d = [\Sym^d V] = [D^d V], \quad h'_d = [T_p\Sym^d V] = [T_p D^d V]=\sum_{\substack{a_1+\cdots+a_n=d \\ 0\leq a_i<p}}x_1^{a_1}\cdots x_n^{a_n}.\]
Notice that $h_d=h'_d$ when $d<p$, and $e_d=0$ for $d>n$. We make the convention $e_d=h_d=0$ for $d<0$. We also consider the \defi{Schur functions} $s_\ll = [\bb{S}_{\ll}V]$ and corresponding \defi{modular Schur functions} $s'_\ll$ obtained by replacing $h_d$ with $h'_d$ in the Jacobi-Trudi identity (see \cite{walker} for a more detailed discussion). The following identities (which the reader can take as definitions) will suffice for our discussion: for $a\geq b\geq 0$, we have
\[ s_{(a,b)} = h_a\cdot h_b-h_{a+1}\cdot h_{b-1},\quad s'_{(a,b)} = h'_a\cdot h'_b-h'_{a+1}\cdot h'_{b-1}.\]

\begin{theorem}\label{thm:chars-d<2p}
 Suppose that $p\leq d<2p$ and that $e\geq d-1$. We have
 \[ \left[H^1(\PP,D^d\mc{R}(e))\right] = s'_{(e+p,d-p)}.\]
\end{theorem}

For an example of the above character formula, note that when $d=p$ and $e=(n-1)p-n$ we get
\[s'_{(e+p,d-p)} = s'_{(n(p-1),0)} = h'_{n(p-1)} = (x_1\cdots x_n)^{p-1}=1,\]
so $H^1(\PP,D^d\mc{R}(e))$ is the trivial $\SL$-module as shown in the special case $t=k=1$ of Theorem~\ref{thm:corner-cohomology}. The extremal case $e=d-1$ was considered in \cite{liu-polo}*{Proposition~1.3.1}, where it was shown that the respective cohomology groups are irreducible $\SL$-representations. We will give an alternative explanation of irreducibility based on \cites{doty-walker,walker}, and explain why this property may fail for more general values of $e$ (see Example~\ref{ex:reducible-coh}).

\medskip

\noindent{\bf The case $n=3$.} Generalizing Theorem~\ref{thm:chars-d<2p} to arbitrary $d,e$ remains an open problem for a general $n$. We specialize further our discussion to the case $n=3$ and aim to give a recursive formula for the characters
\begin{equation}\label{eq:def-h0-h1-de} 
h^0(d,e) = \left[H^0(\PP,D^d\mc{R}(e-1)) \right]\quad\text{and}\quad h^1(d,e)=\left[H^1(\PP,D^d\mc{R}(e-1)) \right]\quad\text{ when }d,e\geq 0, 
\end{equation}
as well as an explicit non-recursive formula when the characteristic of the underlying field is $p=2$. The recursion is originally due to Linyuan Liu \cite{liu-thesis}*{Th\'eor\`eme~1}, and we merely offer a different perspective from a classical algebraic geometry point of view. Our choice of notation \eqref{eq:def-h0-h1-de} is motivated in part by \eqref{eq:Hj-X-vs-P}, but also by trying to make the role of the parameters $d,e$ appear more symmetric: if we write $\PP^{\vee}$ for $\bb{P}(V^{\vee})$ and $X^{\vee}$ for the incidence variety obtained by exchanging $V$ with $V^{\vee}$ then we get using \eqref{eq:Hj-X-vs-P} and Serre duality
\[\begin{aligned} H^j(\PP,D^d\mc{R}(e-1)) &= H^{j+n-2}(X,\mc{O}_X(e,-d-2)) = H^{j+n-2}(X^{\vee},\mc{O}_{X^{\vee}}(-d-2,e)) \\
&= H^{n-1-j}(X^{\vee},\mc{O}_{X^\vee}(d,-e-2))^{\vee} = H^{1-j}(\PP^{\vee},D^e\mc{R}(d-1))^{\vee}.
\end{aligned}
\]
Since substituting $V$ with $V^{\vee}$ and then dualizing the resulting representation does not change characters, we conclude that
\[ h^0(d,e)=h^1(e,d) \text{ for all }d,e\geq 0.\]
We need some more notation in order to state our result. We first generalize the notion of modular Schur functions to allow for truncations along iterates of Frobenius. We define the \defi{$q$-truncated complete symmetric functions $h^{(q)}_d$}  and \defi{$q$-truncated Schur functions $s^{(q)}_{(a,b)}$} via
\[ h^{(q)}_d = \sum_{\substack{a_1+\cdots+a_n=d \\ 0\leq a_i<q}}x_1^{a_1}\cdots x_n^{a_n},\qquad s^{(q)}_{(a,b)} = h^{(q)}_a\cdot h^{(q)}_b - h^{(q)}_{a+1}\cdot h^{(q)}_{b-1},\]
observing that $h^{(p)}_d=h'_d$ and $s^{(p)}_{(a,b)}=s'_{(a,b)}$.
We also define
\[ \chi(d,e) = h^0(d,e)-h^1(d,e),\]
which is the $\SL$-equivariant Euler characteristic of $D^d\mc{R}(e-1)$, and is given as in characteristic zero by
\[\chi(d,e) = \begin{cases}
s_{(e-1,d)} & \text{if }e>d,\\
-s_{(d-1,e)} & \text{if }d>e,\\
0 & \text{otherwise.}
\end{cases}\]
It follows that it is enough to recursively compute $h^1(d,e)$, which we do as follows -- note that $h^1(0,e)=0$ for all $e$, so we will assume that $d>0$.

\begin{theorem}\label{thm:recursion-sl3}
 Suppose that $n=3$, and let $1\leq t<p$ and $q=p^k$ such that $tq\leq d<(t+1)q$. We have
 \begin{itemize}
    \item If $e>(t+1)q-2$ then 
    \[h^1(d,e)=0.\]
    \item If $e< tq$ then 
    \[h^1(d,e)=s_{(d-1,e)}.\]
    \item If $tq\leq e\leq(t+1)q-2$ then
    \begin{equation}\label{eq:main-rec-sl3}
    \begin{aligned}
    h^1(d,e)=F^q(h_t^{\vee})\cdot h^1(d-tq,e-tq) &+ F^q(h_{t-1}^{\vee})\cdot s^{(q)}_{(e-1+(2-t)q,d-tq)}\\
    &+F^q(h_{t-2}^{\vee})\cdot h^0(q(t+1)-d-2,q(t+1)-e-2)^{\vee}.
    \end{aligned}
    \end{equation}
 \end{itemize}
\end{theorem}

Observe that when $p\leq d\leq e\leq 2p-2$ we get $t=1$ and $q=p$, so Theorem~\ref{thm:recursion-sl3} yields
\[h^1(d,e) = F^p(h_1^{\vee})\cdot h^1(d-p,e-p) + s^{(p)}_{(e-1+p,d-p)} = s^{(p)}_{(e-1+p,d-p)},\]
since $h^1(d-p,e-p)=0$ follows by taking $t=d-p$ and $q=1$ in the first case of the recursion. In particular, Theorem~\ref{thm:recursion-sl3} leads to the same conclusion as Theorem~\ref{thm:chars-d<2p} in this special case.

It is elementary to translate Theorem~\ref{thm:recursion-sl3} into a recursive relation for the cohomology of line bundles on $X$, and we leave the translation to the interested reader. As a first application, we describe the (lexicographically) highest weight in the representation $H^1(\PP,D^d\mc{R}(e-1))$: every weight for $\SL_3$ is of the form 
\[\ll=(a,b,0)=(a-b)\omega_1+b\omega_2, \]
and the \defi{lexicographic order} on such weights is defined by
\[ (a_1,b_1,0)>(a_2,b_2,0)\quad\text{ if }\quad a_1>a_2\quad\text{ or }\quad (a_1=a_2\text{ and }b_1>b_2).\]

\begin{theorem}\label{thm:sl3-highest-weight-h1}
 Suppose that $n=3$, $d,e\geq 0$, and that $h^1(d,e)\neq 0$. 
 \begin{enumerate}
     \item If $d>e$ then the (lexicographically) highest weight in $h^1(d,e)$ is 
     \[(d-1-e)\omega_1+e\omega_2 = (d-1,e,0).\]
     \item If $d\leq e$, then there exists $q'=p^s$, $m>0$, satisfying
     \[ d = mq'+d',\quad e=mq'+e',\quad p\nmid m, \quad 0\leq d',e'\leq q'-2.\]
     If $q'$ is minimal with the above property, then the highest weight in $h^1(d,e)$ is
     \[d'\omega_1+(e-2e'-2)\omega_2 = (d-e'-2,e-2e'-2,0).\]
 \end{enumerate}
\end{theorem}

When $d>e$, the conclusion of Theorem~\ref{thm:sl3-highest-weight-h1} agrees with the characteristic zero situation, when we have $h^1(d,e)=s_{(d-1,e)}$. We illustrate the case $d\leq e$ by comparing it to Theorem~\ref{thm:chars-d<2p}: if $p\leq d\leq e\leq 2p-2$ then
\[ h^1(d,e) = s'_{(e-1+p,d-p)} = h'_{e-1+p}\cdot h_{d-p} - h'_{e+p}\cdot h_{d-p-1}.\]
Using the identifications $h'_f = (h'_{3p-3-f})^{\vee}$ (see Section~\ref{subsec:trunc-div-sym}) and $h'_f=h_f$ for $f<p$, we get
\[ h^1(d,e) = h_{2p-2-e}^{\vee}\cdot h_{d-p} - h_{2p-3-e}^{\vee}\cdot h_{d-p-1},\]
whose highest weight is $(2p-2-e)\omega_2+(d-p)\omega_1$. With the notation in Theorem~\ref{thm:sl3-highest-weight-h1}, we have $d=d'+p$, $e=e'+p$, hence $d'=d-p$ and $e-2e'-2=e-2(e-p)-2=2p-2-e$, so the two results agree.

\medskip

\noindent {\bf A non-recursive character formula and Nim.} We write $d=(d_k\cdots d_0)_2$ for the $2$-adic expansion of a non-negative integer $d$:
\[ d= \sum_{i=0}^k d_i\cdot 2^i,\text{ with }d_i\in\{0,1\}\text{ for all }i.\]
The \defi{Nim-sum} (or \defi{bitwise xor}) $a\oplus b$ is defined by performing addition modulo $2$ to each of the digits in the $2$-adic expansion of non-negative integers $a,b$:
\[ a\oplus b = c\text{ if and only if }a_i + b_i \equiv c_i \text{ mod }2\text{ for all }i.\]
We note that the above relation can be expressed more symmetrically as $a\oplus b\oplus c=0$, and that triples $(a,b,c)$ satisfying it correspond to winning positions in the three-pile game of Nim. We define the (trivariate) \defi{Nim character functions} to be the elements $\mc{N}_m\in A=\bb{Z}[x_1,x_2,x_3]/\langle x_1x_2x_3-1\rangle$ given as
\begin{equation}\label{eq:def-nim}
 \mc{N}_m = \sum_{\substack{a+b+c=2m \\ a\oplus b\oplus c = 0}} x_1^a x_2^b x_3^c.
\end{equation}
For instance, we have $\mc{N}_2 = x_1^2x_2^2+x_1^2x_3^2+x_2^2x_3^2$,
\[\mc{N}_3 = x_1^3x_2^3+x_1^3x_3^3+x_2^3x_3^3+x_1^3x_2^2x_3+x_1^2x_2^3x_3+x_1^3x_2x_3^2+x_1x_2^3x_3^2+x_1^2x_2x_3^3+x_1x_2^2x_3^3 = \mc{N}_1\cdot\mc{N}_2,\]
and it can be shown that $\mc{N}_m$ is the character of the simple $\SL_3$-module of highest weight $m\omega_2$ (see Lemma~\ref{lem:rec-Nim}).

\begin{theorem}\label{thm:sl3-p=2-characters}
 Suppose that $n=3$, $p=2$, $k\geq 1$ and $2^k\leq d\leq e\leq 2^{k+1}-2$. Write $d=(d_k\cdots d_0)_2$, and for each $i=1,\cdots,k$ consider the left and right truncations of the binary expansion of $d$:
 \[ l_i(d) = (d_k\cdots d_{i+1})_2\quad\text{ and }\quad r_i(d) = (d_{i-1}\cdots d_0)_2.\]
 If we define $l_i(e),r_i(e)$ analogously, and set $I=\{i:1\leq i\leq k,\ d_i=e_i=1,\ l_i(d)=l_i(e),\ r_i(e)\leq 2^i-2\}$ then
 \[h^1(d,e) = \sum_{i\in I} F^{2^{i+1}}\left(\mc{N}_{l_i(d)}\right)\cdot s^{(2^i)}_{(r_i(e)-1+2^{i+1},r_i(d))}.\]
 %If we consider the set $I=\{i: 1\leq i\leq k,\ d_i=1,\text{ and }r_i\leq 2^i-2+d-e\}$ then we have
 %\[
 % h^1(d,e)=\sum_{i\in I}  F^{2^{i+1}}\left(\mc{N}_{l_i}\right)\cdot s_{(2^i-2+d-e,r_i)}.
 %\]
\end{theorem}

Theorem~\ref{thm:sl3-p=2-characters} is highly suggestive of a layered structure on cohomology, represented (in each of the regions II, III, V, VI) by triangles of a different color in Figure~\ref{fig:PicX n=3,p=2}. The layers are parametrized by $i\geq 1$, have size $2^i-1$ and are arranged periodically with period $2^{i+1}$. They encode characters given by truncated Schur functions (the ``building blocks"), tensored with Frobenius twists of appropriate Nim polynomials (the ``multiplicities"), where the building blocks remain constant, but the multiplicities grow with the distance from the origin. In Theorem~\ref{thm:sl3-p=2-characters} we always have $k\in I$. If we let $q=2^k$ then $l_k(d)=0$, $r_k(d)=d-q$ and $r_k(e)-1+2^{k+1}=e-1+q$, so the ``top-layer" contribution to $h^1(d,e)$ is $s^{(q)}_{(e-1+q,d-q)}$ (compare with Theorem~\ref{thm:chars-d<2p}).
%\iffalse
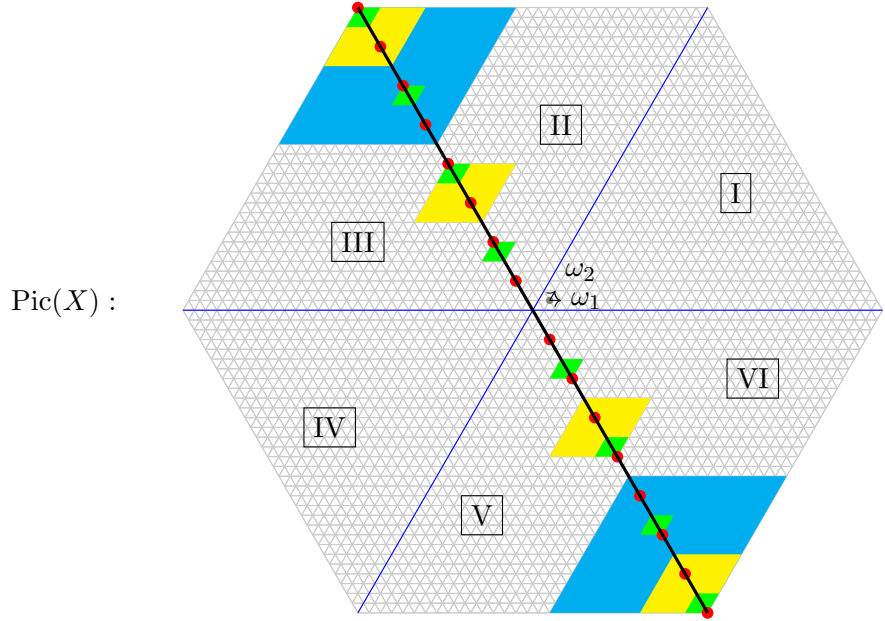
\begin{figure}[H]
\[\op{Pic}(X):\qquad
\begin{tikzpicture}[baseline=(O.base)]
\setlength\weightLength{.15cm}%{.22cm}
\begin{rootSystem}{A}
\weightLattice{31}
\node (O) at \weight{0}{0} {};
\wt{1}{1}
\draw[->] \weight{1}{1} -- \weight{2}{1} 
    node[right]{$\omega_1$};
\draw[->] \weight{1}{1} -- \weight{1}{2} 
    node[above right]{$\omega_{2}$};
%\draw[very thick] \weight{-20}{0} -- \weight{20}{0};
%\draw[very thick] \weight{-20}{1} -- \weight{19}{1};
\draw[blue!100] \weight{-31}{0} -- \weight{31}{0};
\draw[blue!100] \weight{0}{-31} -- \weight{0}{31};
\fill[cyan!100] \weight{17}{-17} -- \weight{31}{-17} -- \weight{31}{-31} --
\weight{17}{-31} -- cycle;
\fill[yellow!100] \weight{9}{-9} -- \weight{15}{-9} -- \weight{15}{-15} --
\weight{9}{-15} -- cycle;
\fill[yellow!100] \weight{25}{-25} -- \weight{31}{-25} -- \weight{31}{-31} --
\weight{25}{-31} -- cycle;
\fill[green!100] \weight{5}{-5} -- \weight{7}{-5} -- \weight{7}{-7} -- \weight{5}{-7} -- cycle;
\fill[green!100] \weight{13}{-13} -- \weight{15}{-13} -- \weight{15}{-15} -- \weight{13}{-15} -- cycle;
\fill[green!100] \weight{21}{-21} -- \weight{23}{-21} -- \weight{23}{-23} -- \weight{21}{-23} -- cycle;
\fill[green!100] \weight{29}{-29} -- \weight{31}{-29} -- \weight{31}{-31} -- \weight{29}{-31} -- cycle;
\filldraw[red!100] \weight{3}{-3} circle (2pt);
\filldraw[red!100] \weight{7}{-7} circle (2pt);
\filldraw[red!100] \weight{11}{-11} circle (2pt);
\filldraw[red!100] \weight{15}{-15} circle (2pt);
\filldraw[red!100] \weight{19}{-19} circle (2pt);
\filldraw[red!100] \weight{23}{-23} circle (2pt);
\filldraw[red!100] \weight{27}{-27} circle (2pt);
\filldraw[red!100] \weight{31}{-31} circle (2pt);

\fill[cyan!100] \weight{-17}{17} -- \weight{-31}{17} -- \weight{-31}{31} --
\weight{-17}{31} -- cycle;
\fill[yellow!100] \weight{-9}{9} -- \weight{-15}{9} -- \weight{-15}{15} --
\weight{-9}{15} -- cycle;
\fill[yellow!100] \weight{-25}{25} -- \weight{-31}{25} -- \weight{-31}{31} --
\weight{-25}{31} -- cycle;
\fill[green!100] \weight{-5}{5} -- \weight{-7}{5} -- \weight{-7}{7} -- \weight{-5}{7} -- cycle;
\fill[green!100] \weight{-13}{13} -- \weight{-15}{13} -- \weight{-15}{15} -- \weight{-13}{15} -- cycle;
\fill[green!100] \weight{-21}{21} -- \weight{-23}{21} -- \weight{-23}{23} -- \weight{-21}{23} -- cycle;
\fill[green!100] \weight{-29}{29} -- \weight{-31}{29} -- \weight{-31}{31} -- \weight{-29}{31} -- cycle;
\filldraw[red!100] \weight{-3}{3} circle (2pt);
\filldraw[red!100] \weight{-7}{7} circle (2pt);
\filldraw[red!100] \weight{-11}{11} circle (2pt);
\filldraw[red!100] \weight{-15}{15} circle (2pt);
\filldraw[red!100] \weight{-19}{19} circle (2pt);
\filldraw[red!100] \weight{-23}{23} circle (2pt);
\filldraw[red!100] \weight{-27}{27} circle (2pt);
\filldraw[red!100] \weight{-31}{31} circle (2pt);

\node[draw] at \weight{12}{12} {I};
\node[draw] at \weight{-7}{19} {II};
\node[draw] at \weight{-19}{7} {III};
\node[draw] at \weight{-12}{-12} {IV};
\node[draw] at \weight{6}{-21} {V};
\node[draw] at \weight{23}{-7} {VI};
\draw[very thick] \weight{-31}{31} -- \weight{31}{-31};
\end{rootSystem}
\end{tikzpicture}
\]
\caption{The Picard group of the incidence correspondence $n=3,p=2$}
\label{fig:PicX n=3,p=2}
\end{figure}
%\fi

A conjectural generalization of the character formula in Theorem~\ref{thm:sl3-p=2-characters} for $p=2$ and all $n$, based on $n$-variate Nim polynomials, is given in \cite{gao}*{Section~4.4} and is verified for $n=4$ in \cite{gao}*{Section~4.3}. The formulas there, along with Theorem~\ref{thm:chars-d<2p} and additional computational evidence obtained via Macaulay2 \cite{GS}, suggest that truncated Schur functors are the ``appropriate" building blocks for the cohomology of line bundles on the incidence correspondence for all $p$. A definitive formula however remains elusive for now.

\section{Preliminaries}\label{sec:prelims}

Throughout this paper, $p>0$ denotes a prime number, $q$ denotes a power of $p$, and $\kk$ is an algebraically closed field of characteristic~$p$. We fix a $\kk$-vector space $V$ with $\dim(V)=n$, and write $\GL=\GL(V)$ for the general linear group of invertible transformations of $V$, and $\SL=\SL(V)$ for the special linear group.

\subsection{Basic polynomial functors}\label{sec:poly-functors}

For a $\kk$-vector space $E$ of dimension $r$, and a non-negative integer $d$, we let $D^d E$ denote the \defi{$d$-th divided power} of $E$, let $S^d E=\Sym^d E$ denote the \defi{$d$-th symmetric power} of~$E$, and let $\bw^d E$ denote the \defi{$d$-th exterior power} of $E$ (see \cite{AFPRW}*{Section~3.1} for a quick review of standard notation and conventions). In addition, if $q=p^k$ then we consider the \defi{$q$-th Frobenius power} of $E$
\[F^q E = \langle e^q :e\in E\rangle_{\kk} \subset S^q E,\]
where $\langle-\rangle_{\kk}$ denotes the $\kk$-linear span. Note that the superscript in $F^q$ refers to the degree of the polynomial functor as opposed to the number of iterates of Frobenius.

The functors $D^d, S^d, \bw^d, F^q$ extend naturally if we replace $E$ with a locally free sheaf $\mc{E}$ of rank $r$ on a $\kk$-variety $Y$, and if we write $\mc{E}^{\vee} = \ShHom_{\mc{O}_Y}(\mc{E},\mc{O}_Y)$ for the dual of $\mc{E}$, then we have natural identifications
\begin{equation}\label{eq:S-D-E-dual}
\left(S^d\mc{E}\right)^{\vee} \simeq D^d\left(\mc{E}^{\vee}\right),\quad \left(\bw^d\mc{E}\right)^{\vee}\simeq \bw^d(\mc{E}^{\vee}),\quad .
\end{equation}
that we will use freely throughout this article. We will also write
\[ \det(\mc{E})=\bw^r\mc{E},\text{ and }\Sym(\mc{E}) = \bigoplus_{d\geq 0} S^d\mc{E}\]
for the \defi{determinant} of $\mc{E}$ and the \defi{symmetric algebra} of $\mc{E}$ respectively.

If $\varphi:Y\lra Y$ denotes the absolute Frobenius morphism on $Y$, then $F^p\mc{E}=\varphi^*\mc{E}$, and $F^q=\left(\varphi^k\right)^*$. In particular, $F^q$ commutes with tensor operations:
\begin{equation}\label{eq:Fq-tensor-ops} 
F^q(\mc{E}^{\oo d}) = (F^q\mc{E})^{\oo d},\ F^q(S^d\mc{E})=S^d(F^q\mc{E}),\ F^q(D^d\mc{E})=D^d(F^q\mc{E}),\ \left(F^q\mc{E}\right)^{\vee} \simeq F^q\left(\mc{E}^{\vee}\right),
\end{equation}
and for every exact complex $\mc{E}_{\bullet}$ of locally free sheaves, we have that $F^q\mc{E}_{\bullet}$ is also exact. If we dualize the inclusion $F^q\left(\mc{E}^{\vee}\right)\hookrightarrow S^q\left(\mc{E}^{\vee}\right)$ and apply the identifications in \eqref{eq:S-D-E-dual}, \eqref{eq:Fq-tensor-ops}, then we get a natural quotient map
\begin{equation}\label{eq:D-onto-F}
    D^q\mc{E} \onto F^q\mc{E}.
\end{equation} 

\subsection{Truncated divided and symmetric powers}\label{subsec:trunc-div-sym}

We recall some basic properties of truncations defined via Frobenius -- for more information, the reader may consult \cite{doty-walker}, \cite{sun}*{Section~3}. Suppose that $\mc{E}$ is a locally free sheaf of rank $r$, and let $q=p^k$ for some $k\geq 1$. We define the \defi{$q$-truncated divided powers} of $\mc{E}$ via
\begin{equation}\label{eq:q-truncated-divided-E}
 T_qD^d\mc{E} = \ker\left(D^d\mc{E} \lra D^{d-q}\mc{E} \oo F^q\mc{E}\right),
\end{equation}
obtained from the natural inclusion $D^d\mc{E}\hookrightarrow D^{d-q}\mc{E}\oo D^q\mc{E}$ followed by the quotient map induced by \eqref{eq:D-onto-F}. Similarly, we define the \defi{$q$-truncated symmetric powers} of $\mc{E}$ via
\begin{equation}\label{eq:q-truncated-symmetric-E}
 T_qS^d\mc{E} = \coker\left(S^{d-q}\mc{E} \oo F^q\mc{E} \lra S^d\mc{E}\right).
\end{equation}

When $q=p$, we have isomorphisms
\[ T_pD^d\mc{E} \simeq T_pS^d\mc{E}\text{ for all }d\geq 0,\]
and $T_pD^dV = T_pS^d V$ is an irreducible $\SL$- (or $\GL$-) representation.

When $d<q$, we have $T_qD^d\mc{E}=D^d\mc{E}$ and $T_qS^d\mc{E}=S^d\mc{E}$. Moreover, if we let $N_q(\mc{E})=r(q-1)$ then we have
\[T_qD^d\mc{E}=0=T_qS^d\mc{E}\text{ for }d>N_q(\mc{E})\quad\text{ and }\quad T_qD^{N_q(\mc{E})}\mc{E} \simeq T_qS^{N_q(\mc{E})}\mc{E} = \det(\mc{E})^{q-1}.\]
Moreover, we have a perfect pairing
\[ T_qS^d\mc{E} \oo T_qS^{N_q(\mc{E})-d}\mc{E} \lra T_qS^{N_q(\mc{E})}\mc{E},\]
which induces a natural isomorphism
\begin{equation}\label{eq:duality-TqS-TqD}
T_qS^d\mc{E} \simeq T_qD^{N_q(\mc{E})-d}\left(\mc{E}^{\vee}\right)\oo \det(\mc{E})^{q-1}.
\end{equation} 
Since the sheaf of graded algebras $\Sym\mc{E}$ is locally free over the subalgebra $\Sym(F^q\mc{E})$ (a sheaf version of the familiar fact that a polynomial ring $A[x_1,\cdots,x_r]$ is free over $A[x_1^q,\cdots,x_r^q]$), and the quotient
\[ \Sym\mc{E} / \langle F^q\mc{E}\rangle = T_q\Sym\mc{E} = \bigoplus_{d\geq 0}T_qS^d\mc{E},\]
it follows by the graded Nakayama's lemma (after restricting to degree $d$) that $\Sym^d\mc{E}$ admits a filtration
\[ \Sym^d\mc{E} = \mc{E}^0 \supseteq \mc{E}^1 \supseteq \cdots\]
where $\mc{E}^i/\mc{E}^{i-1} \simeq F^q(S^i\mc{E}) \oo T_q S^{d-iq}\mc{E}$. Dually, we have a filtration
\begin{equation}\label{eq:filtration-DdE}
\mc{E}_0 \subseteq \mc{E}_1 \subseteq \cdots \subseteq D^d\mc{E}
\end{equation}
with $\mc{E}_i/\mc{E}_{i-1}\simeq F^q(D^i\mc{E}) \oo T_qD^{d-iq}\mc{E}$.

\subsection{Castelnuovo--Mumford regularity}\label{subsec:CM-reg}

The reader may consult \cite{lazarsfeld-pos-I}*{Section~I.8} and \cite{eis-geom-syz}*{Chapter~4} for more details on Castelnuovo--Mumford regularity. We recall the definition \eqref{eq:def-CM-reg} and make the convention that the regularity of the $0$ sheaf is $-\infty$. Regularity behaves well with respect to tensor products,
\begin{equation}\label{eq:reg-tensor}
 \reg(\mc{E}\oo\mc{F}) \leq \reg(\mc{E})+\reg(\mc{F}),
\end{equation}
and with respect to summands,
\begin{equation}\label{eq:reg-summands}
 \reg(\mc{E}\oplus\mc{F}) = \max\{\reg(\mc{E}),\reg(\mc{F})\},
\end{equation}
but less so with respect to other multilinear operations (see Section~\ref{sec:divided-taut-sub} for an interesting example).

We have that $\reg(\mc{O}_{\PP}(-a))=a$ for all $a\in\bb{Z}$, and that any $m$-regular sheaf $\mc{F}$ admits a resolution
\[\cdots \lra \bigoplus \mc{O}_{\PP}(-m-i) \lra \cdots \lra \bigoplus \mc{O}_{\PP}(-m-1) \lra \bigoplus \mc{O}_{\PP}(-m) \lra \mc{F} \lra 0\]
by direct sum of line bundles. Conversely, if $\mc{F}$ admits a resolution $\mc{F}_{\bullet}$ where each $\mc{F}_i$ is $(m+i)$-regular, then $\mc{F}$ is $m$-regular. We can restate this somewhat more generally as follows.

\begin{lemma}\label{lem:reg-from-resolution}
 Suppose that $\mc{F}_{\bullet}$ is a resolution of $\mc{F}$, where $\mc{F}_i$ is $m_i$-regular for all $i$. If we let
 \[m = \max\{m_i-i:i\geq 0\},\] %may need to replace with 0<=i<=dim(P)
 then $\mc{F}$ is $m$-regular.
\end{lemma}

\begin{proof} From the definition of $m$, it follows that $m+i\geq m_i$ for all $i$. Since $\mc{F}_i$ is $m_i$-regular, it is then also $(m+i)$-regular, hence $\mc{F}$ is $m$-regular by the previous paragraph.
\end{proof}

%\[\mc{F}_{\bullet}:\quad \cdots \lra \mc{F}_i \lra \cdots \lra \mc{F}_1 \lra \mc{F}_0 \lra \mc{F} \lra 0,\]

\subsection{Projective bundles, flag varieties, higher direct images \cite{BCRV}*{Sections~9.1--9.3}}\label{subsec:bundles}

If $\mc{E}$ is a locally free sheaf on a variety $Y$, we let $\bb{P}_Y(\mc{E}) = \ul{\Proj}_Y(\Sym\mc{E})$ denote the \defi{projective bundle} associated to $\mc{E}$ as in \cite{hartshorne}*{Section~II.7}. When $Y=\op{Spec}(\kk)$ and $\mc{E}=E$ is a $\kk$-vector space, we obtain the projective space $\bb{P}E$, parametrizing one-dimensional quotients of $E$ (note that this differs from the dual convention often used in the literature where $\bb{P}E$ is defined as $\op{Proj}(\Sym(E^{\vee}))$). If $\pi:\bb{P}_Y(\mc{E})\lra Y$ denotes the structure morphism, then we have a tautological exact sequence
\begin{equation}\label{eq:gen-taut-ses} 0 \lra \mc{R}_{\mc{E}} \lra \pi^*\mc{E} \lra \mc{O}_{\bb{P}_Y(\mc{E})}(1) \lra 0,
\end{equation}
where $\mc{R}_{\mc{E}}$ is the \defi{universal subsheaf}, isomorphic to $\Omega_{\pi}(1)$, the twist by $\mc{O}_{\bb{P}_Y(\mc{E})}(1)$ of the relative cotangent sheaf of the morphism~$\pi$ (that is, \eqref{eq:gen-taut-ses} is obtained from the relative Euler sequence by tensoring with $\mc{O}_{\bb{P}_Y(\mc{E})}(1)$).

We will be mainly interested in the case when $Y=\bb{P}V=\PP$, with tautological short exact sequence \eqref{eq:taut-ses-PV}, and $\mc{E}=\mc{R}^{\vee}$. In that case $X=\bb{P}_{\PP}(\mc{R}^{\vee})$ is the incidence correspondence from the introduction: this is because that \eqref{eq:taut-ses-PV} induces a surjection $V^{\vee}\oo\mc{O}_{\PP}\onto \mc{R}^{\vee}$, which in turn gives rise to a closed immersion
\[ X=\bb{P}_{\PP}(\mc{R}^{\vee}) \hookrightarrow \bb{P}_{\PP}(V^{\vee}\oo\mc{O}_{\PP}) \simeq \PP \times \bb{P}V^{\vee} = \bb{P}V\times\bb{P}V^{\vee},\]
providing the concrete realization \eqref{eq:def-X} for $X$. If we write $\pi:X\lra\PP$ for the structure morphism then we can reconcile the notation $\mc{O}_X(a,b)$ in the introduction with that of \eqref{eq:gen-taut-ses} via
\[ \mc{O}_X(1,0) = \pi^*(\mc{O}_{\PP}(1))\quad\text{ and }\quad\mc{O}_X(0,1) = \mc{O}_X(1).\]
By the adjunction formula, the canonical line bundle on $X$ is given by
\begin{equation}\label{eq:omega-X}
\omega_X = \mc{O}_X(-n+1,-n+1).
\end{equation}
Using $\det(\mc{R}) \simeq \det(V) \oo \mc{O}_{\PP}(-1)\simeq \mc{O}_{\PP}(-1)$, it follows for instance from \cite{hartshorne}*{Ex.~III.8.3,~III.8.4} that for $d\geq 0$, $e\in\bb{Z}$ we have
\[ D^d\mc{R}(e) \simeq R^{n-2}\pi_*(\mc{O}_X(e+1,-d-n+1)),\quad\text{ and }\quad R^j\pi_*(\mc{O}_X(e+1,-d-n+1))=0\text{ for }j\neq n-2.\]
By the Leray spectral sequence, this gives an isomorphism
\begin{equation}\label{eq:cohom-P-vs-incvar} H^i(\PP,D^d\mc{R}(e)) \simeq H^{i+n-2}(X,\mc{O}_X(e+1,-d-n+1)).
\end{equation}
The reader may check that the above isomorphism is $\SL$-equivariant, and in order to make it $\GL$-equivariant one needs to tensor the left side $H^i(\PP,D^d\mc{R}(e))$ with $\det(V)=\bw^n V$. The same argument shows that for $a\in\bb{Z}$ and $-n+1\leq b\leq -1$, one has
\[ R^j\pi_*(\mc{O}_X(a,b)) = 0\text{ for all $j$, hence }H^i(X,\mc{O}_X(a,b))=0\text{ for all }i,\]
and by interchanging the roles of $V$ and $V^{\vee}$, the same vanishing holds when $-n+1\leq a\leq -1$ and $b\in\bb{Z}$.

%Using Serre duality and the symmetry of $V$ and $V^\vee$, we get the following abstract vector spaces to be isomorphic (have the same dimension)
%\begin{align*}
%    H^1(\bb{P}V,D^d\mc{R}(e))&\simeq H^{n-1}(X,\mc{O}_X(e+1,-d-n+1))\\
%    &\simeq H^{n-1}(X,\mc{O}_X(-d-n+1,e+1))\\
%    &\simeq H^{n-2}(X,\mc{O}_X(d,-e-n))\\
%    &\simeq H^0(\bb{P}V,D^{e+1}\mc{R}(d-1)).
%\end{align*}
%To get the correct equivariance (an isomorphism which is $\GL$-equivariant), one can trace through above identifications and check that if we denote by $\mc{H}_{d,e}$ the functor assigning to $V$ the space $H^1(\bb{P}V,D^d\mc{R}(e))$ then we have a natural identification
%\[ \mc{H}_{d,e}(V) \simeq \mc{H}_{e+1,d-1}(V^{\vee})^{\vee}.\] 

\subsection{Symmetric powers of the tautological subsheaf}\label{subsec:sym-pows-R}

In Section~\ref{sec:divided-taut-sub} we will discuss the problem of computing the regularity of divided powers of the sheaf $\mc{R}$ on the projective space $\PP$. Here we show that the Kempf--Weyman geometric technique gives a simple, characteristic-free answer (likely well-known to experts), in the case of symmetric powers.

\begin{theorem}\label{thm:sym-pows-R}
 For $a\geq 0$ there exists a resolution $\mc{F}_{\bullet}$ of $\Sym^a\mc{R}$ with
 \[ \mc{F}_i = \bb{S}_{(a,a,1^i)}V \oo \mc{O}_{\PP}(-a-i)\quad\text{ for }\quad i=0,\cdots,n-2.\]
 As a consequence, we have $H^1(\PP,\Sym^a\mc{R}(a-2))=\bb{S}_{(a-1,a-1)}V$ for $a\geq 1$, and $\reg(\Sym^a\mc{R})=a$.
\end{theorem}

\begin{proof}
 We write $S=\Sym(V)$ for the homogeneous coordinate ring of $\PP$, and consider the graded $S$-module
 \[ M = \bigoplus_{d\geq 0} H^0(\PP,\Sym^a\mc{R}(d+a)),\]
 where the summand indexed by $d$ lies in degree $d$. The sheaf associated to $M$ is $\mc{M}=\Sym^a\mc{R}(a)$. By \cite{weyman}*{Theorem~5.1.2}, $M$ admits a free resolution $F_{\bullet}$ where
\[ F_i = \bigoplus_{j\geq 0} H^j\left(\PP,\bw^{i+j}\mc{R}\oo\mc{M}\right) \oo S(-i-j).\]
Using the short exact sequence
\[ 0 \lra \bb{S}_{(a,1^{i+j})}\mc{R} \lra \bw^{i+j}\mc{R} \oo \Sym^a\mc{R} \lra \bb{S}_{(a+1,1^{i+j-1})}\mc{R} \lra 0,\]
which after twisting by $\mc{O}_{\PP}(a)$ leads to a short exact sequence 
\[ 0 \lra \bb{S}_{(a,1^{i+j})}\mc{R}(a) \lra \bw^{i+j}\mc{R} \oo \mc{M} \lra \bb{S}_{(a+1,1^{i+j-1})}\mc{R}(a) \lra 0.\]
Using for instance \cite{BCRV}*{Corollary~9.8.6}, it follows that $\bb{S}_{(a,1^{i+j})}\mc{R}(a)$ has vanishing higher cohomology, and 
\[ H^0\left(\PP,\bb{S}_{(a,1^{i+j})}\mc{R}(a)\right) = \bb{S}_{(a,a,1^{i+j})}V.\]
Applying \cite{BCRV}*{Theorems~9.4.1 and~9.8.5, Lemma~9.3.4}, we get that the cohomology of $\bb{S}_{(a+1,1^{i+j-1})}\mc{R}(a)$ is identically zero. This implies that
\[ F_i = \bb{S}_{(a,a,1^i)}V \oo S(-i) \text{ for }i \geq 0,\]
and since $\dim(V)=n$, the only non-zero terms occur when $i\leq n-2$. Sheafifying the resolution of $M$ and twisting by $\mc{O}_{\PP}(-a)$ yields the desired resolution of $\Sym^a\mc{R}$.

For the final assertions, we suppose that $a\geq 1$ and use the hypercohomology spectral sequence
\[ E_1^{-i,j} = H^j(\PP,\mc{F}_i(a-2)) \Longrightarrow \bb{H}^{j-i}(\PP,\mc{F}_{\bullet}(a-2)) = H^{j-i}(\PP,\Sym^a\mc{R}(a-2)).\]
The only non-vanishing term with $i<j$ on the $E_1$ page occurs for $i=n-2$ and $j=n-1$, and we have
\[ E_1^{-n+2,n-1} = H^{n-1}(\PP,\mc{O}_{\PP}(-n)) \oo \bb{S}_{(a,a,1^{n-2})}V = \det(V)^{\vee} \oo \bb{S}_{(a,a,1^{n-2})}V = \bb{S}_{(a-1,a-1)}V,\]
which implies the desired description for $H^1(\PP,\Sym^a\mc{R}(a-2))$. Moreover, since $H^1(\PP,\Sym^a\mc{R}(a-2))\neq 0$, we conclude $\reg(\Sym^a\mc{R})\geq a$, while the existence of the resolution $\mc{F}_{\bullet}$ combined with Lemma~\ref{lem:reg-from-resolution} yields $\reg(\Sym^a\mc{R})\leq a$. This shows that $\reg(\Sym^a\mc{R})=a$ when $a\geq 1$. For $a=0$, $\Sym^a\mc{R}=\mc{O}_{\PP}$ has regularity $0$, concluding the proof.
\end{proof}

\section{Divided powers of the tautological subsheaf}\label{sec:divided-taut-sub}

Throughout this section, $V$ is a vector space of dimension $n$, and $\PP = \bb{P}V \simeq \PP^{n-1}$ is the projective space parametrizing $1$-dimensional quotients of $V$ (or equivalently, $1$-dimensional subspaces of $V^{\vee}$). We consider the tautological sequence \eqref{eq:taut-ses-PV}, where $\mc{Q}=\mc{O}_{\PP}(1)$ is the tautological line bundle on $\PP$, and $\mc{R}\simeq\Omega^1_{\PP}(1)$ is the tautological subsheaf, with $\rk(\mc{R})=n-1$. The goal of this section is to prove the following.

\begin{theorem}\label{thm:reg-Dd-R}
 If $n\geq 3$, $q=p^k$ for some $k\geq 0$, and $t\cdot q\leq d<(t+1)\cdot q$ for some $1\leq t<p$, then
 \[\reg\left(D^d\mc{R}\right) = (t+n-2)\cdot q - n+2.\]
\end{theorem}

\noindent Notice that Theorem~\ref{thm:reg-Dd-R} gives the second part of Theorem~\ref{thm:FV-to-PV} (equation \eqref{eq:reg-DdR}), while the first part (equation \eqref{eq:Hj-X-vs-P}) was explained in Section~\ref{subsec:bundles}. Using the general results from Section~\ref{subsec:bundles} we will also explain how to derive Theorems~\ref{thm:main} and~\ref{thm:corner-cohomology}.

We begin with some basic observations regarding the statement of Theorem~\ref{thm:reg-Dd-R}.
\begin{enumerate}
\item For every positive integer $d$ there exists a unique pair $(t,k)$ with $1\leq t<p$, $k\geq 0$, such that if we let $q=p^k$ then $t\cdot q\leq d<(t+1)\cdot q$. 
\item One has that $t\cdot q$ is the leading term in the base $p$ expansion of $d$, and Theorem~\ref{thm:reg-Dd-R} asserts that $\reg(D^d\mc{R})$ only depends on this term.
\item When $d<p$, the formula in Theorem~\ref{thm:reg-Dd-R} reduces to $\reg(D^d\mc{R})=d$, which follows from the isomorphism $D^d\mc{R}\simeq\Sym^d\mc{R}$ and Theorem~\ref{thm:sym-pows-R}.
\item Over a field of characteristic zero, the isomorphism $D^d\mc{R}\simeq\Sym^d\mc{R}$ holds for all $d\geq 0$, hence $\reg(D^d\mc{R})=d$, in contrast to the positive characteristic situation.
\item The hypothesis $n\geq 3$ is necessary for the conclusion of Theorem~\ref{thm:reg-Dd-R}: if $n=2$ then $\mc{R}\simeq\mc{O}_{\PP}(-1)$ and therefore $D^d\mc{R}\simeq\mc{O}_{\PP}(-d)$ and $\reg(D^d\mc{R})=d$ again.
\end{enumerate}

Our strategy for proving Theorem~\ref{thm:reg-Dd-R} is to study cohomological aspects of the filtration of $D^d\mc{R}$ from Section~\ref{subsec:trunc-div-sym}, whose composition factors are
\[ F^q(D^a\mc{R}) \oo T_qD^{d-aq}\mc{R},\quad a=0,\cdots,t.\]
We will show that for $e$ sufficiently large, the only composition factor of $D^d\mc{R}(e)$ that can carry higher cohomology is $(F^q(D^t\mc{R}) \oo T_qD^{d-tq}\mc{R})(e)$, and detect precisely when cohomology vanishing occurs for this factor. We begin by establishing a series of preliminary results.

\begin{lemma}\label{lem:coh2-higher-DdR}
 If $q=p^k$ for some $k\geq 0$, then for all $a\geq 0$, $j\geq 2$, and $e\geq -n+1$, we have
 \[ H^j\left(\PP,F^q(D^a\mc{R})\oo \mc{O}_{\PP}(e)\right) = 0.\]
\end{lemma}

\begin{proof} The tautological sequence \eqref{eq:taut-ses-PV} induces for each $a\geq 0$ a short exact sequence
\begin{equation}\label{eq:ses-DdR}
 0\lra D^a\mc{R} \lra D^aV \oo \mc{O}_{\PP} \lra D^{a-1}V\oo\mc{O}_{\PP}(1) \lra 0.
\end{equation}
If we apply $F^q$ to \eqref{eq:ses-DdR} and tensor with $\mc{O}_{\PP}(e)$, then we get a short exact sequence
\[0 \lra F^q(D^a\mc{R}) \oo \mc{O}_{\PP}(e) \lra F^q(D^aV) \oo \mc{O}_{\PP}(e) \lra F^q(D^{a-1}V)\oo\mc{O}_{\PP}(q+e) \lra 0.\]
% Michael: I added tensor by \mc{O}_{\PP}(e) in this short exact sequence.
Considering the induced long exact sequence in cohomology, the desired conclusion follows from the fact that 
\[H^j\left(\PP,\mc{O}_{\PP}(e)\right) = 0\]
for all $j\geq 1$ and $e\geq -n+1$.
\end{proof}

Using the sequence \eqref{eq:ses-DdR} with $q=1$ and $a=d\geq 1$, we find that
\[ H^1(\PP,D^d\mc{R}(-1)) = D^{d-1}V \neq 0,\]
and in particular $\reg(D^d\mc{R})>0$. It follows from the definition of regularity and Lemma~\ref{lem:coh2-higher-DdR} that for $d\geq 1$
\begin{equation}\label{eq:reg-DdR-H1} \reg(D^d\mc{R}) = \max\{m : H^1(\PP,D^d\mc{R}(m-2))\neq 0\}.
\end{equation}
The next observation is that the regularity of $D^d\mc{R}$ is non-decreasing as a function of $d$.

\begin{lemma}\label{lem:reg-Dr-increases}
 We have $\reg(D^d\mc{R})\geq\reg(D^{d-1}\mc{R})$ for all $d\geq 1$. 
\end{lemma}

\begin{proof}
Since $\reg(\mc{R})=1\geq \reg(\mc{O}_{\PP})=0$, we may assume that $d\geq 2$. The tautological sequence \eqref{eq:taut-ses-PV} gives rise to an exact complex
\begin{equation}\label{eq:exact-div-wedge} 0 \lra D^{d+n-1}\mc{R} \lra D^{d+n-2}\mc{R} \oo V \lra \cdots \lra D^d\mc{R} \oo \bw^{n-1}V \lra D^{d-1}\mc{R} \oo \bw^n V \lra 0.
\end{equation}
We let $m=\reg(D^{d-1}\mc{R})$ and let $e=m-2\geq -n+1$. If we twist \eqref{eq:exact-div-wedge} by $\mc{O}_{\PP}(e)$ then we know by Lemma~\ref{lem:coh2-higher-DdR} that each term can only have non-vanishing cohomology in $H^0$ and $H^1$. Applying the hypercohomology spectral sequence to \eqref{eq:exact-div-wedge} (or using a simple induction based on the fact that for any short exact sequence 
\[ 0 \lra \mc{A} \lra \mc{B} \lra \mc{C} \lra 0,\]
if $\mc{A}, \mc{B}$ have vanishing cohomology beyond $H^0,H^1$, then so does $\mc{C}$) we conclude that the induced map in cohomology
\[ H^1\left(\bb{P}V,  D^d\mc{R} \oo \mc{O}_{\PP}(e) \oo \bw^{n-1}V\right) \lra H^1\left(\bb{P}V,  D^{d-1}\mc{R} \oo\mc{O}_{\PP}(e) \oo \bw^n V\right)\]
is surjective. In particular, since $H^1(\bb{P}V,D^{d-1}\mc{R} \oo \mc{O}_{\PP}(e)) \neq 0$ we get $H^1(\bb{P}V,D^{d}\mc{R} \oo \mc{O}_{\PP}(e)) \neq 0$, so $\reg(D^d\mc{R})\geq m$.
\end{proof}

\begin{lemma}\label{lem:reg-Fq-Da-R}
 If $1\leq a<p$ and $q=p^k$ for some $k\geq 0$, then we have
 \[\reg(F^q(D^a\mc{R})) = q\cdot(n+a-2)-n+2.\]
 Moreover, we have
 \[ H^1\left(\PP,F^q(D^a\mc{R})\oo\mc{O}_{\PP}(q\cdot(n+a-2)-n) \right) = F^q(\bb{S}_{(a-1,a-1)}V).\]
\end{lemma}

\begin{proof} 
The assumption $a<p$ implies that $D^a\mc{R}\simeq S^a\mc{R}$. Since the Frobenius power functor $F^q$ is exact, it follows from Theorem~\ref{thm:sym-pows-R} that $F^q(D^a\mc{R})$ admits a resolution $\tilde{\mc{F}}_{\bullet}$ with
\begin{equation}\label{eq:tilde-F-res-FqDaR} \tilde{\mc{F}}_i = F^q(\bb{S}_{(a,a,1^i)}V) \oo \mc{O}_{\PP}(-q(a+i))\quad\text{ for }\quad i=0,\cdots,n-2.
\end{equation}
Applying Lemma~\ref{lem:reg-from-resolution}, we conclude that $F^q(D^a\mc{R})$ is $m$-regular, where
\[ m = \max\{q(a+i)-i:i=0,\cdots,n-2\} = q\cdot(n+a-2)-n+2.\]
To show that $\reg(F^q(D^a\mc{R}))=m$, it suffices to prove that $H^1(\PP,F^q(D^a\mc{R})(m-2))\neq 0$, which we compute exactly as in the proof of Theorem~\ref{thm:sym-pows-R}. We have a hypercohomology spectral sequence
\[ E_1^{-i,j} = H^j(\PP,\tilde{\mc{F}}_i(m-2)) \Longrightarrow \bb{H}^{j-i}(\PP,\tilde{\mc{F}}_{\bullet}(m-2)) = H^{j-i}(\PP,F^q(D^a\mc{R})(m-2)).\]
Since $m-2-q(a+i)\geq -n$ for $i=0,\cdots,n-2$, with equality if and only if $i=n-2$, it follows that the only non-vanishing term with $i<j$ on the $E_1$ page occurs for $i=n-2$ and $j=n-1$, and we have
\[ H^1(\PP,F^q(D^a\mc{R})(m-2)) = E_1^{-n+2,n-1} = F^q(\bb{S}_{(a-1,a-1)}V),\]
as desired.
\end{proof}

\begin{lemma}\label{lem:vanishing-FqDr-tensor-G}
 With notation as in Lemma~\ref{lem:coh2-higher-DdR}, if $\mc{G}$ is an $m$-regular sheaf on $\PP$ and $e\geq m$, then
 \[ H^j\left(\PP,F^q(D^a\mc{R}) \oo \mc{G}(e)\right) = 0\text{ for all }j\geq 2.\]
\end{lemma}

\begin{proof} Consider a resolution $\mc{G}_{\bullet}$ of $\mc{G}$, with $\mc{G}_i = \bigoplus \mc{O}_{\PP}(-m-i)$, and let $\mc{F}_{\bullet}$ denote the induced resolution of $F^q(D^a\mc{R}) \oo \mc{G}(e)$, with
\[\mc{F}_i = F^q(D^a\mc{R}) \oo \mc{G}_i(e)\text{ for }i\geq 0.\]
If $e\geq m$, then it follows from Lemma~\ref{lem:coh2-higher-DdR} that
\[ H^j(\PP,\mc{F}_i)=0\text{ for }j\geq 2\text{ and }0\leq i\leq n-1.\]
The desired conclusion follows by applying again the hypercohomology spectral sequence.
\end{proof}

To prove vanishing for $H^1$ we will need some stronger hypotheses, as follows.

\begin{lemma}\label{lem:H1vanishing-FqDr-tensor-G}
 If $0\leq a<p$, and $\mc{G}$, $e$ are such that
 \[ H^1(\PP,\mc{G}(e-qa)) = H^2(\PP,\mc{G}(e-q(a+1)) = \cdots = H^{n-1}(\PP,\mc{G}(e-q(a+n-2)) = 0,\]
 then we have
 \[ H^1(\PP,F^q(D^a\mc{R}) \oo \mc{G}(e)) = 0.\]
\end{lemma}

\begin{proof}
 If $a=0$ then $F^q(D^a\mc{R}) \oo \mc{G}(e)=\mc{G}(e-qa)$, so the desired vanishing is part of the hypothesis. We assume that $a\geq 1$ and have using \eqref{eq:tilde-F-res-FqDaR} a resolution $\tilde{\mc{F}}_{\bullet}\oo\mc{G}(e)$ of $F^q(D^a\mc{R})\oo\mc{G}(e)$. Applying the hypercohomology spectral sequence
 \[ E_1^{-i,j} = H^j(\PP,\tilde{\mc{F}}_{i}\oo\mc{G}(e)) \Longrightarrow \bb{H}^{j-i}(\PP,\tilde{\mc{F}}_{\bullet}\oo\mc{G}(e)) = H^{j-i}(\PP,F^q(D^a\mc{R})\oo\mc{G}(e)),\]
 it suffices to prove that $E_1^{-i,i+1}=0$ for $i=0,\cdots,n-2$. This is equivalent to $H^{i+1}(\PP,\mc{G}(e-q(a+i))=0$ for $i=0,\cdots,n-2$, which is our hypothesis. 
\end{proof}

For the next results we consider the dual of the tautological sequence \eqref{eq:taut-ses-PV}
\begin{equation}\label{eq:dual-taut-ses-PV}
0 \lra \mc{O}_{\PP}(-1) \lra V^{\vee}\oo\mc{O}_{\PP} \lra \mc{R}^{\vee} \lra 0,
\end{equation}
and analyze regularity properties of several functors applied to $\mc{R}^{\vee}$.

\begin{lemma}\label{lem:Sd-Rvee-0-reg}
  For all $d\geq 0$, we have that $S^d(\mc{R}^{\vee})$ is $0$-regular.
\end{lemma}

\begin{proof}
Dualizing \eqref{eq:ses-DdR} (or taking symmetric powers in \eqref{eq:dual-taut-ses-PV}), we get a $2$-step resolution
\[0\lra S^{d-1}(V^{\vee})\oo \mc{O}_{\PP}(-1) \lra S^d(V^{\vee})\oo \mc{O}_{\PP} \lra S^{d}(\mc{R}^{\vee}) \lra 0,\]
of $S^d(\mc{R}^{\vee})$. By Lemma~\ref{lem:reg-from-resolution}, it follows that $S^d(\mc{R}^{\vee})$ is $0$-regular, as desired.
\end{proof}

\begin{lemma}\label{lem:reg-Fq-wedge-Rvee}
 If $q=p^k$ for some $k\geq 0$, then $F^q\left(\bw^i\mc{R}^{\vee}\right)\simeq\bw^i\left(F^q(\mc{R}^{\vee})\right)$ is $i\cdot(q-1)$-regular for $0\leq i\leq n-1$.
\end{lemma}

\begin{proof}
 It follows from \eqref{eq:dual-taut-ses-PV} that $\bw^i\mc{R}^{\vee}$ has a resolution
 \[ 0\lra\mc{O}_{\PP}(-i)\lra V^{\vee} \oo \mc{O}_{\PP}(-i
 +1)
 % Michael: I changed -i to -i+1
 \lra\cdots\lra \bw^{i-1}V^{\vee}\oo \mc{O}_{\PP}(-1) \lra \bw^i V^{\vee}\oo \mc{O}_{\PP} \lra \bw^i\mc{R}^{\vee} \lra 0.\]
 Applying $F^q(-)$, this yields a resolution $\mc{F}_{\bullet}$ of $F^q\left(\bw^i\mc{R}^{\vee}\right)$ with
 \[\mc{F}_j = F^q\left(\bw^{i-j} V^{\vee}\right)\oo \mc{O}_{\PP}(-q\cdot j),\text{ for }j=0,\cdots,i.\]
 Since $\mc{F}_j$ is $q\cdot j$-regular, and
 \[\max\{q\cdot j-j : j=0,\cdots,i\} = i\cdot(q-1),\]
 the desired conclusion follows from Lemma~\ref{lem:reg-from-resolution}. 
\end{proof}

\begin{lemma}\label{lem:reg-TqSc-Rvee}
 Suppose that $0\leq c<r\cdot q$, where $q=p^k$ for some $k\geq 1$. We have
 \[ \reg\left(T_qS^c\left(\mc{R}^{\vee}\right)\right) \leq (r-1)\cdot(q-2).\]
\end{lemma}

\begin{proof} We consider the Koszul resolution $\mc{K}_{\bullet}$ of $T_qS^c\left(\mc{R}^{\vee}\right)$, given by
\begin{equation}\label{eq:K-res-TqSc}
\mc{K}_i = \bw^i\left(F^q(\mc{R}^{\vee})\right)\oo S^{c-iq}\left(\mc{R}^{\vee}\right),\text{ for }i=0,1,\cdots,r-1.
\end{equation}
It follows from \eqref{eq:reg-tensor}, and Lemmas~\ref{lem:Sd-Rvee-0-reg} and~\ref{lem:reg-Fq-wedge-Rvee} that $\mc{K}_i$ is $i\cdot(q-1)$-regular. Since
\[\max\{ i\cdot(q-1)-i : i=0,\cdots,r-1\} = (r-1)\cdot(q-2),\]
the desired conclusion follows from Lemma~\ref{lem:reg-from-resolution}.
\end{proof}

\begin{corollary}\label{cor:reg-Tq-Db-R}
 If $q=p^k$ for some $k\geq 1$, and $b\geq 0$, then
 \[ \reg(T_qD^b\mc{R}) \leq (n-1)\cdot(q-1)-n+2.\]
\end{corollary}

\begin{proof}
 We let $N_q = N_q(\mc{R}) = (n-1)\cdot(q-1)$, and recall that $T_qD^b\mc{R}=0$ for $b>N_q$, 
 % Michael: I changed \cdots to \cdot
 in which case the conclusion is trivial. If $0\leq b\leq N_q$ and if we let $c=N_q-b$, then it follows from \eqref{eq:duality-TqS-TqD} that
 \begin{equation}\label{eq:iso-TqD-TqS-R}
  T_qD^b\mc{R} \simeq T_qS^c\left(\mc{R}^{\vee}\right) \oo \mc{O}_{\PP}(-q+1).
 \end{equation}
 Since $0\leq c\leq N_q<(n-1)\cdot q$, it follows from Lemma~\ref{lem:reg-TqSc-Rvee} (with $r=n-1$) that
 \[ \reg(T_qD^b\mc{R}) = \reg\left(T_qS^c\left(\mc{R}^{\vee}\right)\right) + (q-1) \leq (n-2)\cdot(q-2)+(q-1) = (n-1)\cdot(q-1)-n+2,\]
 concluding our proof.
\end{proof}

Combining the earlier results, we obtain the following key step in the proof of Theorem~\ref{thm:reg-Dd-R}.

\begin{proposition}\label{prop:reg-FqDa-oo-TqDb}
 If $0\leq a<p$, $b\geq 0$, and $q=p^k$ for some $k\geq 1$, then
\[H^1\left(\PP,F^q(D^a\mc{R}) \oo T_qD^b\mc{R} \oo \mc{O}_{\PP}(e)\right)=0\text{ for }e\geq \max\left(\reg\left(F^q(D^a\mc{R})\right) + q-2,(n-1)\cdot(q-1)\right).\]
\end{proposition}

\begin{proof} We let $N_q=(n-1)\cdot(q-1)$, and note as in the proof of Corollary~\ref{cor:reg-Tq-Db-R} that $T_qD^b\mc{R}=0$ for $b>N_q$. We may therefore assume that $0\leq b\leq N_q$, we let $c=N_q-b<(n-1)\cdot q$, and recall the isomorphism \eqref{eq:iso-TqD-TqS-R}. Letting $r=n-1$ and using the notation \eqref{eq:K-res-TqSc}, we obtain a resolution $\mc{F}_{\bullet}$ of $F^q(D^a\mc{R}) \oo T_qD^b\mc{R}$, with
\[\mc{F}_i = F^q(D^a\mc{R}) \oo \mc{K}_i(-q+1)\text{ for }0\leq i\leq n-2,\ \mc{F}_i=0\text{ for }i\geq n-1.\]
Suppose now that $e\geq \reg\left(F^q(D^a\mc{R})\right) + q-2$. Since $\mc{K}_0=S^b(\mc{R}^{\vee})$
% Michael: I changed S^c to S^b
is $0$-regular by Lemma~\ref{lem:Sd-Rvee-0-reg}, and $\mc{O}_{\PP}(-q+1)$ is $(q-1)$-regular, we get that
\[ e+1\geq \reg\left(F^q(D^a\mc{R})\right) + (q-1) \overset{\eqref{eq:reg-tensor}}{\geq} \reg(\mc{F}_0).\]
It follows that $\mc{F}_0$ is $(e+1)$-regular, hence
\begin{equation}\label{eq:H1F0e=0}
H^1(\PP,\mc{F}_0(e)) \overset{\eqref{eq:def-m-reg}}{=} 0.
\end{equation}
If we assume also that $e\geq(n-1)\cdot(q-1)$, then using the fact that $\reg(\mc{K}_i)\leq i\cdot(q-1)$, observed in the proof of Lemma~\ref{lem:reg-TqSc-Rvee}, it follows that whenever $\mc{K}_i$ is non-zero, we have
\[\reg(\mc{K}_i(-q+1)) = \reg(\mc{K}_i)+(q-1)\leq(i+1)\cdot(q-1)\leq(n-1)\cdot(q-1).\]
We can then apply Lemma~\ref{lem:vanishing-FqDr-tensor-G} with $\mc{G}=\mc{K}_i(-q+1)$ to conclude that
\[ H^j\left(\PP,\mc{F}_i(e)\right)=0\text{ for }j\geq 2\text{ and }i\geq 0.\]
Combining the above vanishing with \eqref{eq:H1F0e=0}, and applying (as in the proof of Lemma~\ref{lem:H1vanishing-FqDr-tensor-G}) the hypercohomology spectral sequence to the resolution $\mc{F}_{\bullet}(e)$ of $F^q(D^a\mc{R}) \oo T_qD^b\mc{R} \oo \mc{O}_{\PP}(e)$, we get the desired conclusion.
\end{proof}

We are now ready to prove the main result of this section.

\begin{proof}[Proof of Theorem~\ref{thm:reg-Dd-R}] We first check by induction on $q$ that $D^d\mc{R}$ is $e$-regular, where
\begin{equation}\label{eq:def-e}
e = (t+n-2)\cdot q - n+2.
\end{equation}
For $q=1$ we have $d=t<p$ and the conclusion follows from $D^d\mc{R}\simeq\Sym^d\mc{R}$ and Theorem~\ref{thm:sym-pows-R}. We will therefore assume that $q>1$. Notice that $e\geq t\geq 0$, hence $e-i\geq -n+1$ for $1\leq i\leq n-1$. We can therefore apply Lemma~\ref{lem:coh2-higher-DdR} to conclude that
\[ H^i\left(\PP,D^d\mc{R}(e-i)\right)=0\text{ for }i\geq 2.\]
To prove that $D^d\mc{R}$ is $e$-regular, it remains to check that 
\[H^1\left(\PP,D^d\mc{R}(e-1)\right)=0.\]
Since $D^d\mc{R}$ has a filtration with composition factors
\[F^q(D^a\mc{R}) \oo T_qD^{d-aq}\mc{R},\ a=0,\cdots,t,\]
it suffices to check that
\[H^1\left(\PP,F^q(D^a\mc{R}) \oo T_qD^{d-aq}\mc{R}\oo \mc{O}_{\PP}(e-1)\right)=0\text{ for }a=0,\cdots,t.\]
For $0\leq a\leq t-1$, this follows from Proposition~\ref{prop:reg-FqDa-oo-TqDb}: indeed, since $t\geq 1$, we have
\[e-1 \geq (n-1)\cdot q - n+1=(n-1)\cdot (q-1);\]
moreover, for $1\leq a\leq t-1$, it follows from Lemma~\ref{lem:reg-Fq-Da-R} that
\[\reg\left(F^q(D^a\mc{R})\right) + q-2 = q\cdot(n+a-2)-n+q \leq q\cdot(n+t-2)-n \leq e-1,\]
while for $a=0$ we have $\reg\left(F^q(D^a\mc{R})\right) + q-2=q-2\leq tq+(n-2)\cdot(q-1)-1 = e-1$.

When $a=t$, Proposition~\ref{prop:reg-FqDa-oo-TqDb} is no longer sufficient, so we need a different argument. Notice that since $d-tq<q$, we have $T_qD^{d-tq}\mc{R}=D^{d-tq}\mc{R}$. We write $c=d-tq$, $\mc{G}=D^c\mc{R}$ and we verify that
\[H^1\left(\PP,F^q(D^t\mc{R}) \oo \mc{G}(e-1)\right)=0\]
using Lemma~\ref{lem:H1vanishing-FqDr-tensor-G}. We have that for $i=1,2,\cdots,n-2$ that
\[ e-1-q(t+i) = q(n-2-i)-n+1\geq -n+1,\]
hence the vanishing $H^{i+1}(\PP,\mc{G}(e-1-q(t+i))=0$ follows from Lemma~\ref{lem:coh2-higher-DdR}. To prove
$H^{1}(\PP,\mc{G}(e-1-qt))=0$, it suffices to check that
\begin{equation}\label{eq:regG-up-bound} \reg(\mc{G}) \leq e-qt = (n-2)\cdot q-n+2.
\end{equation}
If $c=0$ then there is nothing to prove. If $c>0$, we consider $1\leq t'<p$ and $q'=p^{k'}$ such that $t'q'\leq c<(t'+1)q'$. Since $c<q$, we have $q'\leq q/p$, and we know by induction that
\[ \reg(D^c\mc{R}) = (t'+n-2)q'-n+2 \leq (p-1+n-2)q'-n+2.\]
To prove \eqref{eq:regG-up-bound} it is then enough to check that
\[(p-1+n-2)\cdot\frac{q}{p} \leq (n-2)\cdot q,\quad\text{or equivalently},\quad p-1\leq (n-2)\cdot(p-1),\]
which is true because $n\geq 3$.

To prove that $\reg(D^d\mc{R})=e$, it remains to check that $D^d\mc{R}$ is not $(e-1)$-regular. By Lemma~\ref{lem:reg-Dr-increases}, it suffices to consider the case when $d=tq$, and it is enough to show that
\[ H^1\left(\PP,D^{tq}\mc{R}(e-2)\right)\neq 0.\]
From the discussion in Section~\ref{subsec:trunc-div-sym}, there is a short exact sequence
\begin{equation}\label{eq:ses-F-DtqR-FqDtR} 
0 \lra \mc{F} \lra D^{tq}\mc{R} \lra F^q(D^{t}\mc{R}) \lra 0,
\end{equation}
where $\mc{F}$ has a filtration with composition factors $F^q(D^a\mc{R})\oo T_qD^{(t-a)q}\mc{R}$, for $a=0,\cdots,t-1$. By Lemma~\ref{lem:reg-Fq-Da-R}, 
\[ H^1\left(\PP,F^q(D^{t}\mc{R})(e-2)\right)\neq 0,\]
so it suffices to check that $H^2(\PP,\mc{F}(e-2))=0$, which in turn is implied by
\[ H^2(\PP,F^q(D^a\mc{R})\oo T_qD^{(t-a)q}\mc{R}(e-2))=0\text{ for }a=0,\cdots,t-1.\]
We prove this vanishing by combining Lemma~\ref{lem:coh2-higher-DdR} with $\mc{G}=T_qD^{(t-a)q}\mc{R}$ and Corollary~\ref{cor:reg-Tq-Db-R}. To that end, we need to verify the inequality
\[(t+n-2)q-n \geq (n-1)(q-1)-n+2,\]
which (since $t\geq 1$) is implied by $(n-1)q\geq(n-1)(q-1)+2$, which in turn is a consequence of $n\geq 3$.
\end{proof}

A slight refinement of the argument in the proof of Theorem~\ref{thm:reg-Dd-R} gives the explicit cohomology formula from Theorem~\ref{thm:corner-cohomology}, as follows.

\begin{proof}[Proof of Theorem~\ref{thm:corner-cohomology}]

If we let $q=p^k$, then we have
\[d=tq\quad\text{and}\quad a = (t+n-2)q-n+1 = e-1,\]
where $e$ is as in \eqref{eq:def-e}. The identification
\[H^{n-1}\left(X,\mc{O}_X(a,b)\right) = H^1\left(\PP,D^{d}\mc{R}(a-1)\right)\]
follows from \eqref{eq:cohom-P-vs-incvar} and the fact that $b=-d-n+1$. To prove that the above cohomology groups are given by $F^q(\bb{S}_{(t-1,t-1)}V)$, we analyze the argument in the proof of Theorem~\ref{thm:reg-Dd-R}: using the short exact sequence \eqref{eq:ses-F-DtqR-FqDtR} and $H^2(\PP,\mc{F}(e-2))=0$, we get an exact sequence
\[ H^1(\PP,\mc{F}(e-2)) \lra H^1(\PP,D^{tq}\mc{R}(e-2)) \lra H^1(\PP,F^q(D^t\mc{R})(e-2))\lra 0.\]
Since $D^d\mc{R}(a-1)=D^{tq}\mc{R}(e-2)$, we are reduced via Lemma~\ref{lem:reg-Fq-Da-R} to proving that $H^1(\PP,\mc{F}(e-2))=0$. This can in turn be verified by showing that
\begin{equation}\label{eq:H1-vanishing-FqDj-TqDt-j} H^1(\PP,F^q(D^j\mc{R})\oo T_qD^{(t-j)q}\mc{R}(e-2))=0\text{ for }j=0,\cdots,t-1,
\end{equation}
which we prove by an argument similar to one used in the proof of Theorem~\ref{thm:reg-Dd-R}. We fix $0\leq j\leq t-1$ and let
\[\mc{G} = T_qD^{(t-j)q}\mc{R}.\]
For $i=1,\cdots,n-2$, we have
\[ e-2-q(j+i) \geq e-2-q(j+n-2) = (t-j)q-n \geq q-n \geq -n+1,\]
hence by Lemma~\ref{lem:coh2-higher-DdR} we get
\begin{equation}\label{eq:Hi+1-vanish-Gtwist} H^{i+1}(\PP,\mc{G}(e-2-q(j+i))) = 0 \quad\text{for}\quad i=1,\cdots,n-2.
\end{equation}
Moreover, since $t\geq j+1$, we have
\[ e-2-qj = (t-j+n-2)q-n \geq (n-1)q-n \geq (n-1)(q-1)-n+1\geq\reg(\mc{G})+1,\]
where the last inequality follows from Lemma~\ref{cor:reg-Tq-Db-R}. It follows that
\[ H^1(\PP,\mc{G}(e-2-qj))=0,\]
which combined with \eqref{eq:Hi+1-vanish-Gtwist} shows that Lemma~\ref{lem:H1vanishing-FqDr-tensor-G} applies to give \eqref{eq:H1-vanishing-FqDj-TqDt-j}. The final assertion of the Theorem was discussed in Remark~\ref{rem:simplicity}. 
\end{proof}

We can now use the relation between the cohomology of line bundles on $X$ and that of corresponding vector bundles on $\PP$ to verify the main result of the paper.

\begin{proof}[Proof of Theorem~\ref{thm:main}]

We write $q=p^k$ as usual. We begin by observing that $d=-b-n+1\geq 1$, and that if we let $e=a-1$ then $e\geq d-1\geq 0$. Using \eqref{eq:cohom-P-vs-incvar} we get
\[ H^i(X,\mc{L}) = H^{i-n+2}\left(\PP,D^d\mc{R}(e)\right)\text{ for all }i\in\bb{Z}.\]
By Lemma~\ref{lem:coh2-higher-DdR}, this shows that $H^i(X,\mc{L})=0$ for $i\not\in\{n-2,n-1\}$, proving part (3) of the Theorem. To prove part (1), it suffices to verify that for $e\geq 0$ we have the equivalence
\begin{equation}\label{eq:H1vanish-e-lowerbound} H^1(\PP,D^d\mc{R}(e))=0 \Longleftrightarrow e\geq(t+n-2)q-n+1.
\end{equation}
The implication ``$\Longleftarrow$'' is a direct consequence of Theorem~\ref{thm:reg-Dd-R}. To prove ``$\Longrightarrow$'', suppose by contradiction that $H^1(\PP,D^d\mc{R}(e))= 0$ for some $0\leq e\leq(t+n-2)q-n$. Since $e+1-i>-n$ for $i\leq n-1$, it follows from Lemma~\ref{lem:coh2-higher-DdR} that
\[H^i\left(\PP,D^d\mc{R}(e+1-i)\right) = 0\text{ for }i=2,\cdots,n-1,\]
which implies that $\reg(D^d\mc{R})\leq e+1\leq(t+n-2)q-n+1$, contradicting Theorem~\ref{thm:reg-Dd-R}.

Part (2) is equivalent to the assertion that for $e\geq d-1\geq 0$
\[ H^0(\PP,D^d\mc{R}(e)) = 0 \Longleftrightarrow n=3,\ d = (t+1)q-1,\text{ and }e=d-1,\text{ or if }d=t<p\text{ and }e=d-1.\]
Based on \eqref{eq:ses-DdR}, we can derive a long exact sequence
\[ 0 \lra H^0(\PP,D^d\mc{R}(e)) \lra D^dV \oo \Sym^e V \lra D^{d-1}V \oo \Sym^{e+1}V \lra H^1(\PP,D^d\mc{R}(e)) \lra 0.\]
If $e\geq d$ then
\[ \dim(D^dV \oo \Sym^e V) - \dim(D^{d-1}V \oo \Sym^{e+1}V) = \dim(\bb{S}_{(e,d)}V) > 0,\]
hence $H^0(\PP,D^d\mc{R}(e))$. If $e=d-1$ then $\dim(D^dV \oo \Sym^e V) = \dim(D^{d-1}V \oo \Sym^{e+1}V)$ and therefore
\[\dim\left(H^0(\PP,D^d\mc{R}(e))\right) = \dim\left(H^1(\PP,D^d\mc{R}(e))\right).\]
Using \eqref{eq:H1vanish-e-lowerbound}, the terms above vanish if and only if
\begin{equation}\label{eq:vanishingH1-bdry}
e=d-1\geq (t+n-2)q-n+1.
\end{equation}
Since $d\leq (t+1)q-1$, the above inequality implies $n-3\geq(n-3)q$, which can only happen when $n=3$ or $q=1$. When $q=1$, we have $d=t$ and \eqref{eq:vanishingH1-bdry} holds. When $n=3$, we have that \eqref{eq:vanishingH1-bdry} is equivalent to $d=(t+1)q-1$ and $e=d-1$, concluding our proof.
\end{proof}

We end this section by explaining an approach to non-vanishing of cohomology via the theory of Frobenius splittings as used in the work of Andersen in the 80s. In our work we chose to take a different approach in order to obtain some precise descriptions of cohomology groups (see Theorem~\ref{thm:corner-cohomology}), and we will further refine our arguments in the case $n=3$ in the following section. But if one is only interested in the non-vanishing statement, then perhaps the justification below is more straightforward.

\begin{remark}\label{rem:And-Frob-splittings}
  It follows from \cite{anderson-frob}*{Proposition~3.3} that if $H^i(X,\mc{O}_X(a,b))\neq 0$ for some $i,a,b$ then
  \begin{equation}\label{eq:And-nonvanish}
  H^i\left(X,\mc{O}_X(aq+r,bq+s)\right)\neq 0 \text{ for all }0\leq r,s < q,
  \end{equation}
  where as usual $q=p^k$ is a power of $p=\op{char}(\kk)$. By \eqref{eq:cohom-P-vs-incvar}, \eqref{eq:reg-DdR-H1} and Theorem~\ref{thm:sym-pows-R} we have
\[H^{n-1}\left(X,\mc{O}_X(t-1,-t-n+1)\right)\simeq H^1(\PP,D^t\mc{R}(t-2))\neq 0\text{ for }1\leq t<p.\]
Serre duality (using \eqref{eq:omega-X} and the fact that $\dim(X)=2n-3$) then implies
\[ H^{n-2}\left(X,\mc{O}_X(-t-n+2,t)\right)\simeq H^{n-1}\left(X,\mc{O}_X(t-1,-t-n+1)\right)^{\vee}\neq 0,\]
and \eqref{eq:And-nonvanish} together with Serre duality yields for $0\leq s<q$
\[ 0\neq H^{n-2}\left(X,\mc{O}_X(-(t+n-2)q,tq+s)\right)^{\vee}\simeq H^{n-1}\left(X,\mc{O}_X((t+n-2)q-n+1,-tq-s-n+1)\right).\]
Applying \eqref{eq:cohom-P-vs-incvar} again we get
\[0\neq H^{n-1}\left(X,\mc{O}_X((t+n-2)q-n+1,-tq-s-n+1)\right)\simeq H^1\left(\PP,D^{tq+s}\mc{R}((t+n-2)q-n)\right),\]
which implies
\[ \reg(D^{tq+s}\mc{R}) \geq (t+n-2)q-n+2.\]
The main difficulty in proving Theorem~\ref{thm:reg-Dd-R} is therefore to establish the appropriate cohomology vanishing statements for $D^d\mc{R}(e)$.
\end{remark}

\section{Characters for small values of $d$}\label{sec:small-d}

The goal of this section is to enhance the non-vanishing results of the previous section by computing explicitly the characters of the $\SL$-representations $H^1(\PP,D^d\mc{R}(e))$ for small values of $d$, and giving a proof of Theorem~\ref{thm:chars-d<2p}. Our argument is based on the fundamental exact sequence
\begin{equation}\label{eq:ses-DdR-dsmall-2p}
    0 \lra T_pD^d\mc{R}(e) \lra D^d\mc{R}(e) \lra F^p\mc{R} \oo D^{d-p}\mc{R}(e) \lra 0,
\end{equation}
obtained by dualizing the Koszul resolution from \eqref{eq:K-res-TqSc}, and using the fact that $p\leq d<2p$.

\begin{lemma}\label{lem:TpDdR-vanishing-coh}
 If $e\geq d-1$ then
 \[ H^i(\PP,T_pD^d\mc{R}(e)) = 0 \text{ for }i>0.\]
\end{lemma}

\begin{proof} We recall from Section~\ref{subsec:trunc-div-sym} that $T_pD^d\mc{R}\simeq T_pS^d\mc{R}$, and use the analogue of \eqref{eq:ses-DdR-dsmall-2p} to resolve $T_pS^d\mc{R}$ on the left, obtaining an exact sequence
\[0 \lra F^p\mc{R} \oo S^{d-p}\mc{R}(e) \lra S^d\mc{R}(e) \lra T_pS^d\mc{R}(e) \lra 0.\]
Since $e\geq d-1$, it follows from Theorem~\ref{thm:sym-pows-R} that $S^d\mc{R}(e)$ has no higher cohomology, so it suffices to prove that
\begin{equation}\label{eq:vanishing-FpR-Sd-pR}
H^i(\PP,F^p\mc{R} \oo S^{d-p}\mc{R}(e)) = 0\text{ for }i\geq 2.
\end{equation}
This is a direct application of Lemma~\ref{lem:vanishing-FqDr-tensor-G}, using the fact that $\mc{G} = S^{d-p}\mc{R}$ is $(d-p)$-regular, together with the inequality $e\geq d-p$.
\end{proof}

We record one more elementary character calculation before explaining the proof of Theorem~\ref{thm:chars-d<2p}.

\begin{lemma}\label{lem:coh-twists-Dtr}
 If $0\leq t<p$ then the following are equivalent
 \begin{itemize}
     \item $H^j(\PP,S^t\mc{R}(-f))\neq 0$ for some $j\geq 2$.
     \item $j=n-1$ and $f\geq n$.
 \end{itemize} 
 Moreover, if these conditions are satisfied then
 \[ \left[H^{n-1}(\PP,S^t\mc{R}(-f))\right] = h_t\cdot h_{f-n}^{\vee} - h_{t-1}\cdot h_{f-1-n}^{\vee}. \]
\end{lemma}

\begin{proof}
 We use the identification $S^t\mc{R}=D^t\mc{R}$ and \eqref{eq:ses-DdR} to get an exact sequence
 \[ 0 \lra D^t\mc{R}(-f) \lra D^tV(-f) \lra D^{t-1}V(-f+1)\lra 0.\]
 Since $D^tV(-f)$ and $D^{t-1}V(-f+1)$ can only have cohomology in degree $0$ and $n-1$, it follows by the long exact sequence in cohomology that $H^j(\PP,S^t\mc{R}(-f))$ can be non-zero only when $j=n-1$ and $-f\leq -n$. Moreover, by Serre duality one has $H^{n-1}(\PP,\mc{O}_{\PP}(-f))=(\Sym^{f-n}V)^{\vee}$, hence we get a short exact sequence
 \[0 \lra H^{n-1}(\PP,S^t\mc{R}(-f)) \lra D^tV\oo (\Sym^{f-n}V)^{\vee} \lra D^{t-1}V \oo (\Sym^{f-1-n}V)^{\vee} \lra 0,\]
 from which the character formula follows. That $H^{n-1}(\PP,S^t\mc{R}(-f))$ is indeed non-zero for $f\geq n$ follows from the above by a dimension calculation, but the reader familiar with computations with symmetric functions can also check that $h_t\cdot h_{f-n}^{\vee} - h_{t-1}\cdot h_{f-1-n}^{\vee}=s_{(f-n+t,(f-n)^{n-2})}$ is in fact given by a single Schur function.
\end{proof}

We are now ready to prove the main result of this section.

\begin{proof}[Proof of Theorem~\ref{thm:chars-d<2p}]
It follows from Lemma~\ref{lem:TpDdR-vanishing-coh} and \eqref{eq:ses-DdR-dsmall-2p} that we have an identification
\[H^1(\PP,D^d\mc{R}(e)) = H^1(\PP,F^p\mc{R} \oo S^{d-p}\mc{R}(e)),\]
so our focus will be to compute the cohomology group on the right. Note that by Theorem~\ref{thm:reg-Dd-R} we may assume that $e\leq (n-1)p-n$, otherwise the groups above vanish identically, and so does
\[ s'_{(e+p,d-p)} = h'_{e+p}\cdot h_{d-p}-h'_{e+p+1}\cdot h_{d-p-1},\]
since $h'_i = 0$ for $i>N_p(V) = n(p-1)$ (see Section~\ref{subsec:trunc-div-sym}).

The special case $a=1$ of Theorem~\ref{thm:sym-pows-R} gives a left resolution of $\mc{R}$ by direct sum of line bundles. We apply the Frobenius functor $F^p$ and tensor with $S^{d-p}\mc{R}(e)$ to obtain a resolution $\tilde{\mc{F}}_\bullet$ of $F^p\mc{R} \oo S^{d-p}\mc{R}(e)$, with
\[ \tilde{\mc{F}}_i = F^p\left(\bw^{i+2}V\right) \oo S^{d-p}\mc{R}(e-(i+1)p),\quad\text{for}\quad i=0,\cdots,n-2.\]
By Lemma~\ref{lem:coh-twists-Dtr}, the sheaves $\tilde{\mc{F}}_i$ can only have cohomology in degrees $0,1,n-1$, and we know that $H^1(\PP,\tilde{\mc{F}}_0)=0$ by Theorem~\ref{thm:sym-pows-R} and the inequality $e-p\geq d-1-p$. We consider as before the hypercohomology spectral sequence
\[ E_1^{-i,j} = H^j(\PP,\tilde{\mc{F}}_i) \Longrightarrow  H^{j-i}(\PP,F^p\mc{R} \oo S^{d-p}\mc{R}(e)),\]
and observe that the only non-zero terms in the region $j>i$ are among $E_1^{-i,n-1}$. Combining this with \eqref{eq:vanishing-FpR-Sd-pR}, we conclude that there is an exact complex
\[ 0 \lra H^1(\PP,F^p\mc{R} \oo S^{d-p}\mc{R}(e)) \lra E_1^{-n+2,n-1} \lra E_1^{-n+3,n-1} \lra \cdots \lra E_1^{0,n-1}\lra 0,\]
and therefore
\[ \left[H^1(\PP,F^p\mc{R} \oo S^{d-p}\mc{R}(e))\right] = \sum_{i=0}^{n-2} (-1)^i \cdot \left[ E_1^{-n+2+i,n-1}\right]. \]
Using the identification $e_{n-i} = e_i^{\vee}$ and Lemma~\ref{lem:coh-twists-Dtr} with $t=d-p$ and $f=(n-1-i)p-e$, we get
\begin{equation}\label{eq:char-H1-dsmall-interm} \left[H^1(\PP,F^p\mc{R} \oo S^{d-p}\mc{R}(e))\right] = \sum_{i=0}^{n-2} (-1)^i\cdot F^p(e_i^{\vee})\cdot(h_{d-p}\cdot h_{(n-1-i)p-e-n}^{\vee} - h_{d-p}\cdot h_{(n-1-i)p-e-n-1}^{\vee}).
\end{equation}
Taking characters in the Koszul resolution
\[ \cdots \lra F^p\left(\bw^2 V\right) \oo \Sym^{r-2p}V \lra F^pV \oo \Sym^{r-p}V \lra \Sym^r V \lra T_p\Sym^r V \lra 0,\]
yields the identity 
\[ h'_r = \sum_{i\geq 0}(-1)^i F^p(e_i) \cdot h_{r-ip}. \]
Taking duals allows us to simplify \eqref{eq:char-H1-dsmall-interm} to get
\[\left[H^1(\PP,F^p\mc{R} \oo S^{d-p}\mc{R}(e))\right] = h_{d-p}\cdot(h'_{(n-1)p-e-n})^{\vee}-h_{d-p-1}\cdot(h'_{(n-1)p-e-n-1})^{\vee}.\]
To conclude, we recall from Section~\ref{subsec:trunc-div-sym} that for $N=N_p(V)=n(p-1)$, there is a perfect pairing 
\[ T_pS^r V \oo T_p S^{N-r}V \lra \det(V)^{p-1} \simeq \kk,\]
where the last isomorphism is as $\SL$-representations. Taking characters, we get $(h'_r)^{\vee} = h'_{N-r}$, hence
\[\left[H^1(\PP,F^p\mc{R} \oo S^{d-p}\mc{R}(e))\right] = h_{d-p}\cdot h'_{e+p}-h_{d-p-1}\cdot h'_{e+p} = h'_{d-p}\cdot h'_{e+p}-h'_{d-p-1}\cdot h'_{e+p}=s'_{(e+p,d-p)},\]
concluding our proof.
\end{proof}

We end this section with a discussion of the irreducibility of the cohomology modules from Theorem~\ref{thm:chars-d<2p}, and of the connections of our work with \cite{liu-polo}. Consider the special case when $e=d-1$:
\[ \left[H^1(\PP,D^d\mc{R}(d-1))\right] = s'_{(d+p-1,d-p)}.\] 
The partition $\lambda = (d+p-1,d-p)$ satisfies \defi{Carter's criterion}, meaning that the highest power of $p$ dividing the hook lengths is constant for each column of the Young diagram of $\lambda$. If we write $l_{\mu}$ for the character of the irreducible $\SL$-module of highest weight $\mu$, then \cite{walker}*{Theorem~4.4} and \cite{doty-walker}*{Corollary~2.11} imply that $s'_{\ll}=l_{\mu}$, where $\mu$ is the partition obtained from $\ll$ by sliding as far as possible the nodes in the Young diagram of $\ll$ along a line of slope $1/(p-1)$ in the downward direction (see \cite{james-kerber}*{Section~6.3} for some examples, noting that they use the transposed convention for partitions from ours). In particular, this shows that $H^1(\PP,D^d\mc{R}(d-1))$ is an irreducible $\SL$-module for $p\leq d<2p$, which is proved by different methods in \cite{liu-polo}*{Section~1.3}. Moreover, one has \[ \mu = (d-1,p-1,d-p+1)\text{ if }p\leq d\leq 2p-2,\quad\text{and}\quad \mu = (2p-2,p-1,p-1,1)\text{ if }d=2p-1,\]
recovering the highest weight calculation from \cite{liu-polo}*{Proposition~1.3.1}. The irreducibility of the cohomology modules in Theorem~\ref{thm:chars-d<2p} can fail for general values of $e$, as the following example shows.
\begin{example}\label{ex:reducible-coh}
 Consider the case when $n=4$, $p=5$, $d=6$ and $e=9$. We have $\ll=(e+p,d-p)=(14,1)$, which does not satisfy Carter's criterion: the hook lengths along the first column are~$15$ (divisible by $p$) and~$1$ (not divisible by $p$). By Theorem~\ref{thm:chars-d<2p}, we have
\[ \left[H^1(\PP,D^6\mc{R}(9))\right] = s'_{(14,1)}=l_{(5,4,4,2)}+l_{(4,4,4,3)},\]
where the last equality follows by consulting the database \cite{luebeck}, or the Weyl Modules Gap package \cite{doty-gap}.
\end{example}

\section{Characters for $n=3$}\label{sec:char-n=3}

In this section we specialize our discussion to the case $n=3$, with the goal of proving the recursive description of characters in Theorem~\ref{thm:recursion-sl3}, as well as the characterization of highest weights in Theorem~\ref{thm:sl3-highest-weight-h1} and the non-recursive formula from Theorem~\ref{thm:sl3-p=2-characters} in characteristic two. We write
\[ A = \bb{Z}[x_1,x_2,x_3]/\langle x_1x_2x_3-1\rangle \simeq \bb{Z}[x_1^{\pm 1},x_2^{\pm 1}]\]
and recall that characters of $\SL_3$-representations give rise to functions in $A$, invariant under the action of the symmetric group $\mf{S}_3$ under coordinate permutations. For every such non-zero function $f\in A$, we write $\op{hw}(f)=(a,b,0)=(a-b)\omega_1+b\omega_2\in\bb{Z}^3/\bb{Z}(1,1,1)$ if 
\[f=\gamma\cdot x_1^a\cdot x_2^b + (\text{lower terms}) \quad\in \bb{Z}[x_1^{\pm 1},x_2^{\pm 1}],\]
where $0\neq \gamma\in\bb{Z}$ and the terms are ordered lexicographically. Notice that because of the $\mf{S}_3$-symmetry, $\op{hw}(f)$ must be \defi{dominant}, that is $a\geq b\geq 0$. We have for instance
\[ \op{hw}(s_{(a,b)}) = (a-b)\omega_1+b\omega_2,\quad \op{hw}(h_a) = a\omega_1.\]
In general we will say that $\ll=(a,b,c)$ is a weight in $f$ if $x_1^a x_2^b x_3^c = x_1^{a-c}x_2^{b-c}$ appears with non-zero coefficient in the monomial expansion of $f$.

\begin{lemma}\label{lem:wts<q00}
 If $f\in A$ is $\mf{S}_3$-invariant and $\op{hw}(f)<r\omega_1$ for some $r>0$ then $\op{hw}(f^{\vee})<r\omega_1$.
\end{lemma}

\begin{proof} Suppose by contradiction that $f$ contains the weight $(a,b,0)$ with $a\geq r$, and assume without loss of generality that $a\geq b\geq 0$. By $\mf{S}_3$-symmetry, $f$ also contains $(0,b,a)$, hence $f^{\vee}$ contains $(0,-b,-a)=(a,a-b,0)\geq r\omega_1$, a contradiction.
\end{proof}

For our proofs we will need a number of basic identities in the ring $A$:
\begin{equation}\label{eq:basic-schur} s_{(a,b)}^{\vee} = s_{(a,a-b)},\quad h_a = s_{(a)},\quad h_a^{\vee} = s_{(a,a)}.
\end{equation}
Moreover, for truncated Schur functions we have using results in Section~\ref{subsec:trunc-div-sym}
\begin{equation}\label{eq:basic-trunc} h^{(q)}_a = h_a\text{ for }a<q,\quad \left(h^{(q)}_a\right)^{\vee} = h^{(q)}_{3q-3-a}. \end{equation}

\begin{lemma}\label{lem:trunc-schur=schur}
 Suppose that $q=p^k$, $2q-2\leq a\leq 3q-3$ and $0\leq b<q$. We have
 \[ s^{(q)}_{(a,b)} = s_{(3q-3-a+b,3q-3-a)}.\]
 In particular, if $tq\leq d,e\leq (t+1)q-2$ then
 \[ s^{(q)}_{(e-1+(2-t)q,d-tq)} = s_{(q+d-e-2,(t+1)q-e-2)} = s_{(q+d-e-2,d-tq)}^{\vee}.\]
\end{lemma}

\begin{proof} By definition,
\[ s^{(q)}_{(a,b)} = h^{(q)}_a h^{(q)}_b - h^{(q)}_{a+1} h^{(q)}_{b-1}.\]
Using \eqref{eq:basic-trunc} and $b<q$, we have $h^{(q)}_b=h_b$ and $h^{(q)}_{b-1}=h_{b-1}$. Moreover, since $3q-3-a < q$ we get using \eqref{eq:basic-trunc}
\[ h^{(q)}_a = \left(h^{(q)}_{3q-3-a}\right)^{\vee} = h_{3q-3-a}^{\vee} = s_{(3q-3-a,3q-3-a)},\]
and similarly $h^{(q)}_{a+1} = s_{(3q-4-a,3q-4-a)}=s_{(3q-3-a,3q-3-a,1)}$. It follows that
\[ s^{(q)}_{(a,b)} = s_{(3q-3-a,3q-3-a)}\cdot h_b - s_{(3q-3-a,3q-3-a,1)}\cdot h_{b-1} = s_{(3q-3-a+b,3q-3-a)},\]
where the last equality follows from Pieri's rule. For the last assertion of the lemma, we take $a=e-1+(2-t)q$ and $b=d-tq$, and note that $3q-3-a=(t+1)q-e-2$, so
\[s^{(q)}_{(e-1+(2-t)q,d-tq)} = s_{(q+d-e-2,(t+1)q-e-2)} \overset{\eqref{eq:basic-schur}}{=} s_{(q+d-e-2,d-tq)}^{\vee},\]
as desired.
\end{proof}

For the arguments below, it will also be useful to reinterpret the conditions in part (2) of Theorem~\ref{thm:sl3-highest-weight-h1} in terms of the $p$-adic expansions
\[ d = (d_k\cdots d_0)_p\quad\text{and}\quad e=(e_k\cdots e_0)_p.\]
The existence of $q'$ is equivalent to the fact that for some $s\leq k$ we have
\[ (d_k\cdots d_s)_p = (e_k\cdots e_s)_p =: m,\quad d_s=e_s\neq 0,\quad\text{and}\quad e_i\neq p-1\text{ for some }i<s,\]
in which case $d'=(d_{s-1}\cdots d_0)_p$, $e'=(e_{s-1}\cdots e_0)_p$. The existence of such $s$ is equivalent to the assertion that $tq\leq d,e\leq tq-2$ for $q=p^k$ and some $1\leq t<p$ (that is, the conditions should be satisfied at least for $s=k$).

\begin{proof}[Proof of Theorems~\ref{thm:recursion-sl3} and~\ref{thm:sl3-highest-weight-h1}]
Our strategy is to prove Theorems~\ref{thm:recursion-sl3} and~\ref{thm:sl3-highest-weight-h1} in parallel, by induction on $d$:
\begin{enumerate}[leftmargin=5\parindent]
 \item[{\bf Step 1.}] Verify Theorems~\ref{thm:recursion-sl3} and~\ref{thm:sl3-highest-weight-h1} for $d<p$ (the base case of the induction).
 \item[{\bf Step 2.}] Show that Theorem~\ref{thm:recursion-sl3} for $d$ and Theorem~\ref{thm:sl3-highest-weight-h1} for $\tilde{d}<d$ implies Theorem~\ref{thm:sl3-highest-weight-h1} for $d$.
 \item[{\bf Step 3.}] Show that Theorem~\ref{thm:sl3-highest-weight-h1} for $\tilde{d}<d$ implies Theorem~\ref{thm:recursion-sl3} for $d$.
\end{enumerate}

\medskip

\noindent{\bf Step 1.} Suppose that $tq\leq d<(t+1)q$ in Theorem~\ref{thm:recursion-sl3}. We begin by noting that the vanishing $h^1(d,e)=0$ for $e>(t+1)q-2$ follows from Theorem~\ref{thm:reg-Dd-R}. Similarly, for $e<tq$ the conclusion $h^1(d,e)=s_{(d-1,e)}=-\chi(d,e)$ is equivalent to the vanishing $h^0(d,e)=0$, which in turn is equivalent to $h^1(e,d)=0$. If we let $t'<t$ so that $t'q\leq e<(t'+1)q$ then $d\geq tq > (t'+1)q-2$, and the conclusion follows again from Theorem~\ref{thm:reg-Dd-R}. This proves the first two relations in Theorem~\ref{thm:recursion-sl3} for an arbitrary value of $d$. 

Moreover, if $d<p$ then $t=d$ and $q=1$, so the first two cases in Theorem~\ref{thm:recursion-sl3} are the only ones that can occur ($t=tq\leq e\leq(t+1)q-2=t-1$ is impossible), proving that Theorem~\ref{thm:recursion-sl3} holds when $d<p$. In particular, if $d<p$ then $h^1(d,e)\neq 0$ only when $d>e$, in which case $h^1(d,e)=s_{(d-1,e)}$. This implies 
\begin{equation}\label{eq:hw-h1de-special} 
\op{hw}(h^1(d,e))=\op{hw}(s_{(d-1,e)}) = (d-1,e,0) = (d-1-e)\omega_1+e\omega_2,
\end{equation}
proving Theorem~\ref{thm:sl3-highest-weight-h1} for $d<p$.

\medskip

\noindent{\bf Step 2.} If $e<tq$ then we have $d>e$, and by applying Theorem~\ref{thm:recursion-sl3} we get again \eqref{eq:hw-h1de-special}, as desired. Since $h^1(d,e)\neq 0$, it remains to consider the case
\begin{equation}\label{eq:bounds-for-e}
    tq \leq e \leq (t+1)q-2.
\end{equation} 
By Theorem~\ref{thm:recursion-sl3}, we know that \eqref{eq:main-rec-sl3} holds. Note that
\[ 0\leq q(t+1)-d-2,q(t+1)-e-2\leq q-2,\]
hence by applying the inductive hypothesis for Theorem~\ref{thm:sl3-highest-weight-h1}, we have that $h^0(q(t+1)-d-2,q(t+1)-e-2)=h^1(q(t+1)-e-2,q(t+1)-d-2)$ has highest weight smaller than $q\omega_1$. It follows from Lemma~\ref{lem:wts<q00} that 
\[\op{hw}\left(F^q(h_{t-2}^{\vee})\cdot h^0(q(t+1)-d-2,q(t+1)-e-2)^{\vee}\right) < (t-2)q\omega_2 + q\omega_1 <(t-1)q\omega_2 = \op{hw}\left(F^q(h_{t-1}^{\vee})\right),\]
hence the last term in \eqref{eq:main-rec-sl3} does not contribute to the highest weight of $h^1(d,e)$.

If $d>e$, then $d-tq>e-tq$ and by induction we obtain
\[ \op{hw}\left(F^q(h_t^{\vee})\cdot h^1(d-tq,e-tq)\right)=tq\omega_2 + (d-1-e)\omega_1+(e-tq)\omega_2 = (d-1,e,0).\]
Moreover, combining Lemma~\ref{lem:trunc-schur=schur} with the inequality $q+d-e-2\leq d-2-(t-1)q$, we have
\[ \op{hw}\left(F^q(h_{t-1}^{\vee})\cdot s^{(q)}_{(e-1+(2-t)q,d-tq)} \right) \leq (t-1)q\omega_2 + (d-2-(t-1)q)\omega_2 = (d-2,d-2,0)<(d-1,e,0).\]
It follows that $\op{hw}(h^1(d,e))=(d-1,e,0)$, proving case (1) of Theorem~\ref{thm:sl3-highest-weight-h1}.

Suppose now that $d\leq e$, and let $\tilde{d}=d-tq$, $\tilde{e}=e-tq$. Consider first the case when $h^1(\tilde{d},\tilde{e})\neq 0$, and use the induction hypothesis to write
\[ \tilde{d}=\tilde{m}q'+d',\quad \tilde{e}=\tilde{m}q'+e',\quad p\nmid\tilde{m},\quad 0\leq d',e'\leq q'-2,\]
for $q'=p^s<q$ minimal. Taking $m=\tilde{m}+tq/q'$, we get
\[ d=mq'+d',\quad e=mq'+e',\quad p\nmid m,\quad 0\leq d',e'\leq q'-2,\]
and by comparing the $p$-adic expansions of $d,e$ with those of $\tilde{d},\tilde{e}$, it is clear that $q'$ is minimal for which the above properties hold. Moreover, we have by induction that
\[ \op{hw}\left(F^q(h_t^{\vee})\cdot h^1(\tilde{d},\tilde{e})\right)=tq\omega_2 + d'\omega_1+(\tilde{e}-2e'-2)\omega_2 = d'\omega_1+(e-2e'-2)\omega_2.\]
To see that this must be the highest weight in $h^1(d,e)$, it suffices to note that $q+d-e-2\leq q-2$, hence Lemma~\ref{lem:trunc-schur=schur} implies as before
\[ \op{hw}\left(F^q(h_{t-1}^{\vee})\cdot s^{(q)}_{(e-1+(2-t)q,d-tq)} \right) \leq (t-1)q\omega_2 + (q-2)\omega_2 < tq\omega_2 = \op{hw}\left(F^q(h_t^{\vee})\right) \leq \op{hw}\left(F^q(h_t^{\vee})\cdot h^1(\tilde{d},\tilde{e})\right).\]

Finally, consider the case when $h^1(\tilde{d},\tilde{e})=0$. By \eqref{eq:main-rec-sl3} and Lemma~\ref{lem:trunc-schur=schur} we have
\[
\begin{aligned}
\op{hw}(h^1(d,e)) &= \op{hw}\left(F^q(h_{t-1}^{\vee})\cdot s^{(q)}_{(e-1+(2-t)q,d-tq)} \right) \\
&= (t-1)q\omega_2 + (q+d-e-2,(t+1)q-e-2,0) = (d+tq-e-2,2tq-e-2,0).
\end{aligned}\]
It is then enough to show that $q',m,d',e'$ in part (2) of Theorem~\ref{thm:sl3-highest-weight-h1} are given by $q'=q$, $m=t$, $d'=\tilde{d}$, $e'=\tilde{e}$. Notice that $0\leq \tilde{d}\leq\tilde{e}\leq q-2$ by \eqref{eq:bounds-for-e}. If $\tilde{d}=0$ then $d=tq$ has only one non-zero digit in the $p$-adic expansion, so there is nothing to prove. Suppose $\tilde{d}\geq 1$ and let $1\leq\tilde{t}<p$, $\tilde{q}=p^r$ such that $\tilde{t}\tilde{q}\leq \tilde{d}< (\tilde{t}+1)\tilde{q}$. By Theorem~\ref{thm:main} and Theorem~\ref{thm:FV-to-PV}, we get from $h^1(\tilde{d},\tilde{e})=0$ that $\tilde{e}\geq(\tilde{t}+1)\tilde{q}-1$. If $\tilde{e}=(\tilde{t}+1)\tilde{q}-1$, then all but its leading digit in the $p$-adic expansion are equal to $(p-1)$, while if $\tilde{e}\geq (\tilde{t}+1)\tilde{q}$ then its leading digit differs from that of $\tilde{d}$. Comparing the $p$-adic expansions of $d,e$ with those of $\tilde{d},\tilde{e}$, we conclude that $q'=q$, which then forces $m=t$, $d'=\tilde{d}$, $e'=\tilde{e}$, as desired.

\medskip

\noindent{\bf Step 3.} We have already seen in Step 1 that Theorem~\ref{thm:recursion-sl3} is true when $e>(t+1)q-2$ or $e<tq$. We will therefore assume from now on that \eqref{eq:bounds-for-e} holds. We recall the filtration \eqref{eq:filtration-DdE} of $D^d\mc{R}$, and observe that since $\rk(\mc{R})=2$, we have $N_q(\mc{R})=2(q-1)$ and therefore $T_qD^{d-iq}\mc{R}=0$ if $d-iq > 2q-2$, or equivalently, if $i\leq t-2$. It follows that \eqref{eq:filtration-DdE} reduces to a short exact sequence
\begin{equation}\label{eq:ses-DdR-n=3}
    0 \lra F^q(D^{t-1}\mc{R}) \oo T_qD^{d-(t-1)q}\mc{R} \lra D^d\mc{R} \lra F^q(D^t\mc{R}) \oo D^{d-tq}\mc{R} \lra 0,
\end{equation}
where for the last term we have used $T_qD^{d-tq}\mc{R}=D^{d-tq}\mc{R}$. Our calculation will follow then by twisting \eqref{eq:ses-DdR-n=3} by $\mc{O}_{\PP}(e-1)$ and proving the following:
\begin{equation}\label{eq:Hi-left-term}
\begin{aligned}
    \left[H^1(\PP,F^q(D^{t-1}\mc{R}) \oo T_qD^{d-(t-1)q}\mc{R}(e-1)) \right] &= F^q(h_{t-2}^{\vee})\cdot h^0(q(t+1)-d-2,q(t+1)-e-2)^{\vee}, \\ H^2(\PP,F^q(D^{t-1}\mc{R}) \oo T_qD^{d-(t-1)q}\mc{R}(e-1))&=0.
\end{aligned}
\end{equation}

\begin{equation}\label{eq:H1-right-term}
\begin{aligned}
\left[H^0(\PP,F^q(D^t\mc{R}) \oo D^{d-tq}\mc{R}(e-1)) \right] &= F^q(h_t^{\vee})\cdot h^0(d-tq,e-tq)\\
    \left[H^1(\PP,F^q(D^t\mc{R}) \oo D^{d-tq}\mc{R}(e-1)) \right] &= F^q(h_t^{\vee})\cdot h^1(d-tq,e-tq) + F^q(h_{t-1}^{\vee})\cdot s_{(q+d-e-2,d-tq)}^{\vee}.
\end{aligned}
\end{equation}
and the vanishing of the connecting homomorphism
\begin{equation}\label{eq:connecting=0}
H^0(\PP,F^q(D^t\mc{R}) \oo D^{d-tq}\mc{R}(e-1)) \overset{0}{\lra} H^1(\PP,F^q(D^{t-1}\mc{R}) \oo T_qD^{d-(t-1)q}\mc{R}(e-1)).
\end{equation}

To prove the above properties, we start by noting that the $\SL$-isomorphisms $\bw^3 V\simeq \kk$ and $\bw^2 V\simeq V^{\vee}$ allow us to write the Kosul resolution of $\mc{R}$ as
\begin{equation}\label{eq:Kos-res-R-short}
    0 \lra \mc{O}_{\PP}(-2)\lra V^{\vee}(-1)\lra\mc{R}\lra 0. 
\end{equation}
In what follows we use freely the identification between the divided and symmetric power functors $D^i\simeq S^i$ for $i<p$.

\begin{lemma}\label{lem:Hi-left-term}
    The relations~\eqref{eq:Hi-left-term} hold.
\end{lemma}

\begin{proof} We get from \eqref{eq:Kos-res-R-short} a short exact sequence
\begin{equation}\label{eq:ses-res-St-1-R} 0 \lra S^{t-2}V^{\vee}(-t) \lra S^{t-1}V^{\vee}(-t+1)\lra S^{t-1}\mc{R} \lra 0.
\end{equation}
We also have using \eqref{eq:iso-TqD-TqS-R}, $N_q(\mc{R})=2q-2$ and $(t+1)q-d-2<q$, that
\[T_qD^{d-(t-1)q}\mc{R}\simeq T_qS^{(t+1)q-d-2}\mc{R}^\vee (-q+1)\simeq S^{(t+1)q-d-2}\mc{R}^\vee(-q+1).\]
Applying $F^q$ to \eqref{eq:ses-res-St-1-R} and tensoring with $T_qD^{d-(t-1)q}\mc{R}(e-1)$ we get a short exact sequence
\[0\lra W_1\oo S^{(t+1)q-d-2}\mc{R}^\vee(e-q(t+1)) \lra W_0\oo S^{(t+1)q-d-2}\mc{R}^\vee(e-qt) \lra \mc{F} \lra 0  \]
where $W_1=F^q(S^{t-2}V^{\vee})$, $W_0=F^q(S^{t-1}V^{\vee})$, and $\mc{F} = F^q(D^{t-1}\mc{R}) \oo T_qD^{d-(t-1)q}\mc{R}(e-1)$. Using Lemma~\ref{lem:Sd-Rvee-0-reg} and $e\geq qt$, we obtain
\[ H^1(\PP,W_0\oo S^{(t+1)q-d-2}\mc{R}^\vee(e-qt)) = H^2(\PP,W_0\oo S^{(t+1)q-d-2}\mc{R}^\vee(e-qt)) = 0,\]
hence $H^2(\PP,\mc{F})=0$ and
\[
\begin{aligned}H^1(\PP,\mc{F}) &\simeq W_1\oo H^2(\PP,S^{(t+1)q-d-2}\mc{R}^\vee(e-q(t+1))) \\ &\simeq W_1 \oo H^0(\PP,D^{(t+1)q-d-2}\mc{R}((t+1)q-e-3))^{\vee}
\end{aligned}\]
where the last isomorphism is by Serre duality. Taking characters we get \eqref{eq:Kos-res-R-short}, as desired.
\end{proof}

\begin{lemma}\label{lem:H1-right-term}
    The character formula~\eqref{eq:H1-right-term} holds.
\end{lemma}

\begin{proof} We argue as in the proof of Lemma~\ref{lem:Hi-left-term} to get an exact sequence
\[ 0 \lra W_1 \oo D^{d-tq}\mc{R}(e-1-(t+1)q) \lra W_0 \oo D^{d-tq}\mc{R}(e-1-tq) \lra \mc{F} \lra 0,\]
where $W_1=F^q(S^{t-1}V^{\vee})$, $W_0=F^q(S^{t}V^{\vee})$, and $\mc{F}=F^q(D^t\mc{R})\oo D^{d-tq}\mc{R}(e-1)$. We have
\begin{align*}
    H^j(D^{d-tq}\mc{R}(e-1-(t+1)q))&\overset{\eqref{eq:cohom-P-vs-incvar}}{\simeq} H^{j+1}(X,\mc{O}_{X}(e-(t+1)q,tq-d-2))\\&\simeq H^{j+1}(\bb{F}(V),\mc{O}_{\bb{F}(V)}(e-(t+1)q,0,d-tq+2))\\
    &\overset{\text{Serre duality}}{\simeq} H^{2-j}(\bb{F}(V),\mc{O}_{\bb{F}(V)}((t+1)q-e-2,0,tq-d))^{\vee}.
\end{align*}
Since $(t+1)q-e-2\geq 0\geq tq-d$, the weight $\ll=((t+1)q-e-2,0,tq-d)$ is dominant, and we have using Kempf vanishing \cite{BCRV}*{Theorems~9.4.1,~9.8.5}
\[H^j(\bb{F}(V),\mc{O}_{\bb{F}(V)}(\ll))=\begin{cases}
\bb{S}_{\ll}V & \text{for }j=0,\\
0 &\text{for }j\neq 0.
\end{cases}\]
This shows that
\[ H^2(\PP,D^{d-tq}\mc{R}(e-1-(t+1)q)) = (\bb{S}_{\ll}V)^{\vee}\]
is the only non-vanishing group for the sheaf $D^{d-tq}\mc{R}(e-1-(t+1)q)$, and we know by Lemma~\ref{lem:coh2-higher-DdR} that $H^2(\PP,D^{d-tq}\mc{R}(e-1-tq))=0$. The long exact sequence in cohomology then yields an isomorphism
\[ W_0 \oo H^0(\PP,D^{d-tq}\mc{R}(e-1-tq)) \simeq H^0(\PP,\mc{F})\]
and a short exact sequence
\[ 0 \lra W_0 \oo H^1(\PP,D^{d-tq}\mc{R}(e-1-tq)) \lra H^1(\PP,\mc{F}) \lra W_1 \oo H^2(\PP,D^{d-tq}\mc{R}(e-1-(t+1)q)) \lra 0,\]
from which \eqref{eq:H1-right-term} follows using $s_{\ll}=s_{\ll+(d-tq,d-tq,d-tq)}$.
\end{proof}

\begin{lemma}\label{lem:connecting=0}
    The connecting homomorphism~\eqref{eq:connecting=0} is zero.
\end{lemma}

\begin{proof}
 Since $e-tq\leq q-2$, it follows from Theorem~\ref{thm:sl3-highest-weight-h1} that every dominant weight $\ll=a\omega_1+b\omega_2$ occuring in $h^0(d-tq,e-tq)=h^1(e-tq,d-tq)$ is $q$-restricted, that is, it satisfies $a,b<q$. If we write $L(\mu)$ for the simple $\SL_3$-module of highest weight $\mu$, then we have $F^q(h_t^{\vee}) = L(tq\omega_2)$, and therefore by the Steinberg Tensor Product Theorem \cite{jantzen}*{Proposition~II.3.16} and \eqref{eq:H1-right-term}, we have that every simple module appearing as a composition factor of $H^0(\PP,F^q(D^t\mc{R}) \oo D^{d-tq}\mc{R}(e-1))$ is of the form $L(a\omega_1 + (tq+b)\omega_2)$ with $0\leq a,b<q$.

 Similarly, since $q(t+1)-e-2\leq q-2$, we have by Lemma~\ref{lem:wts<q00} that every dominant weight $\ll$ occuring in $h^0(q(t+1)-d-2,q(t+1)-e-2)^{\vee} = h^1(q(t+1)-e-2,q(t+1)-d-2)^{\vee}$ is $q$-restricted. Arguing as before, it follows from \eqref{eq:Hi-left-term} that every simple composition factor of $H^1(\PP,F^q(D^{t-1}\mc{R}) \oo T_qD^{d-(t-1)q}\mc{R}(e-1))$ is of the form $L(a'\omega_1 + ((t-2)q+b')\omega_2$ for $0\leq a',b'<q$. Since the equality $tq+b=(t-2)q+b'$ cannot hold when $0\leq b,b'<q$, the homomorphism \eqref{eq:connecting=0} must be zero, as desired.
\end{proof}

To conclude the proof of Theorem~\ref{thm:recursion-sl3}, we note that the long exact sequence in cohomology together with \eqref{eq:ses-DdR-n=3}, \eqref{eq:connecting=0}, and the vanishing in \eqref{eq:Hi-left-term}, yields a short exact sequence
\[
\begin{aligned}
0 \lra H^1(\PP,F^q(D^{t-1}\mc{R}) \oo T_qD^{d-(t-1)q}\mc{R}(e-1)) &\lra H^1(\PP,D^d\mc{R}(e-1)) \\
&\lra H^1(\PP,F^q(D^t\mc{R}) \oo D^{d-tq}\mc{R}(e-1)) \lra 0.
\end{aligned}
\]
Taking characters and using \eqref{eq:Hi-left-term} and \eqref{eq:H1-right-term} gives the desired recursion and concludes our proof.
\end{proof}

We conclude our paper with a proof of Theorem~\ref{thm:sl3-p=2-characters}, noting that the recursion \eqref{eq:main-rec-sl3} takes a simpler form when $p=2$: the condition $1\leq t<p$ forces $t=1$, and we get
\begin{equation}\label{eq:main-rec-sl3-p=2}
h^1(d,e)=F^q(s_{(1,1)})\cdot h^1(d-q,e-q) + s^{(q)}_{(e-1+q,d-q)}
\end{equation}
for $q=2^k$, $q\leq d\leq e\leq 2q-2$. We will need the following recursive description of Nim polynomials.

\begin{lemma}\label{lem:rec-Nim}
    If $q=2^k$ and $0\leq m<q$ then $\mc{N}_{q+m} = F^q(s_{(1,1)})\cdot\mc{N}_m$. In particular, $\mc{N}_m = [L(m\omega_2)]$.
\end{lemma}

\begin{proof} Suppose that $a+b+c=2m$, $a\oplus b\oplus c=0$, and consider the base $2$ expansions $a=(a_s\cdots a_0)_2$, $b=(b_s\cdots b_0)_2$, $c=(c_s\cdots c_0)_2$. For each $i=0,\cdots,s$ we have either $a_i=b_i=c_i=0$, or $a_i+b_i+c_i=2$, in which case exactly two of $a_i,b_i,c_i$ are equal to $1$. It follows that if
\[ m = 2^{i_1} + \cdots + 2^{i_r}\text{ for }i_1<\cdots <i_r\]
then $a_i+b_i+c_i=2$ if and only if $i\in\{i_1,\cdots,i_r\}$. A similar argument holds after replacing $m$ by $m+q = m+2^k$, where by the hypothesis we have $k>i_r$. It follows that each solution of the system $a+b+c=2m$, $a\oplus b\oplus c=0$ corresponds to precisely three solutions of the system $A+B+C=2(m+q)$, $A\oplus B\oplus C=0$, namely
\[ (A,B,C) = (a+q,b+q,c),\quad (A,B,C)=(a+q,b,c+q),\quad (A,B,C)=(a,b+q,c+q).\]
Since $F^q(s_{(1,1)})=x_1^qx_2^q+x_1^qx_3^q+x_2^qx_3^q$, we get the desired equality $\mc{N}_{q+m} = F^q(s_{(1,1)})\cdot\mc{N}_m$. Moreover, since $s_{(1,1)} = [L(\omega_2)]$ and $m\omega_2$ is $q$-restricted when $m<q$, the last conclusion follows by induction from the Steinberg Tensor Product Theorem.
\end{proof}

\begin{proof}[Proof of Theorem~\ref{thm:sl3-p=2-characters}] We let $q=2^k$, $d'=d-q$, and $e'=e-q$. If $d'=0$, then $I=\{k\}$, $l_k(d)=0=r_k(d)$, $r_k(e)=e'$, hence Theorem~\ref{thm:sl3-p=2-characters} asserts that
\[ h^1(d,e)=s^{(q)}_{(e'-1+2q,0)}=s^{(q)}_{(e-1+q,d-q)},\]
which follows from \eqref{eq:main-rec-sl3-p=2} and the fact that $h^1(0,e')=0$. We may then assume that $1\leq d'\leq e'$, and we write
\[ d'=(d_r\cdots d_0)_2\quad\text{and}\quad e'=(e_r\cdots e_0)_2\quad\text{ for some }r<k\text{ with }e_r=1.\]
If $d_r=0$ then we get $I=\{k\}$ and $h^1(d,e)=s^{(q)}_{(e-1+q,d-q)}$ as before. If $d_r=1$ then we get $2^r\leq d'\leq e'\leq 2^{r+1}-1$ and we can apply induction to get
 \[h^1(d',e') = \sum_{i\in I'} F^{2^{i+1}}\left(\mc{N}_{l_i(d')}\right)\cdot s^{(2^i)}_{(r_i(e')-1+2^{i+1},r_i(d'))},\]
 where $I'=\{i:1\leq i\leq r,\ d_i=e_i=1,\ l_i(d')=l_i(e'),\ r_i(e')\leq 2^i-2\}$. The formula for $h^1(d,e)$ now follows from the equalities
 \[I=I'\cup\{k\},\quad l_k(d)=0,\quad s^{(2^k)}_{(r_k(e)-1+2^{k+1},r_k(d))}=s^{(q)}_{(e-1+q,d-q)},\quad r_i(d')=r_i(d)\text{ and }r_i(e')=r_i(e)\text{ for }i\in I',
 \] 
 from the identity
 \[ F^q(s_{(1,1)})\cdot F^{2^j}(\mc{N}_m) = F^{2^j}\left(F^{q/2^j}(s_{(1,1)})\cdot\mc{N}_m\right)\]
 and from Lemma~\ref{lem:rec-Nim} combined with the fact that $l_i(d)=l_i(d')+q/2^{i+1}$ for $i\in I'$. Indeed, we have by \eqref{eq:main-rec-sl3-p=2}
 \[
 \begin{aligned}
 h^1(d,e) &= s^{(q)}_{(e-1+q,d-q)} + F^q(s_{(1,1)})\cdot h^1(d',e') \\
 &= s^{(2^k)}_{(r_k(e)-1+2^{k+1},r_k(d))} + \sum_{i\in I'} F^q(s_{(1,1)}) \cdot F^{2^{i+1}}\left(\mc{N}_{l_i(d')}\right)\cdot s^{(2^i)}_{(r_i(e')-1+2^{i+1},r_i(d'))} \\
 &= s^{(2^k)}_{(r_k(e)-1+2^{k+1},r_k(d))} + \sum_{i\in I'}  F^{2^{i+1}}\left(\mc{N}_{l_i(d')+q/2^{i+1}}\right)\cdot s^{(2^i)}_{(r_i(e')-1+2^{i+1},r_i(d'))} \\
&=\sum_{i\in I} F^{2^{i+1}}\left(\mc{N}_{l_i(d)}\right)\cdot s^{(2^i)}_{(r_i(e)-1+2^{i+1},r_i(d))} 
 \end{aligned}\]
 as desired.
\end{proof}

\section*{Acknowledgements}
Experiments with the computer algebra software Macaulay2 \cite{GS} have provided numerous valuable insights. We thank Keller VandeBogert for helpful comments on the manuscript. Raicu acknowledges the support of the National Science Foundation Grant DMS-1901886.

\end{document}